\documentclass[11pt,twocolumn]{article}
\usepackage[margin=0.85in]{geometry}
  \usepackage{xcolor}
\definecolor{refpurple}{RGB}{110,45,145}
\usepackage{float}
\usepackage[
  colorlinks=true,
  linkcolor=refpurple,
  citecolor=refpurple,
  urlcolor=refpurple
]{hyperref}

 \usepackage[utf8]{inputenc}
\usepackage[T1]{fontenc}
\usepackage{amsmath}
\usepackage{amsfonts}
\usepackage{amssymb}

\usepackage{amsthm}
\usepackage{mathtools}
\usepackage{graphicx}
\usepackage{booktabs}
\usepackage{array}
\usepackage{enumitem}
\usepackage{url}
\usepackage{xurl}
 \usepackage{stfloats}
\usepackage{xcolor}
\definecolor{azul}{rgb}{0.0, 0.53, 0.74}
\newcommand{\blueline}{\begin{center}\textcolor{azul}{\rule{150mm}{0.5mm}}\end{center}}
\allowdisplaybreaks

\theoremstyle{plain}
\makeatletter
\newcommand{\safenewtheorem}[3][theorem]{%
  \@ifundefined{#2}{\newtheorem{#2}[#1]{#3}}{}%
}
\@ifundefined{theorem}{\newtheorem{theorem}{Theorem}[section]}{}
\makeatother
\safenewtheorem{lemma}{Lemma}
\safenewtheorem{corollary}{Corollary}
\safenewtheorem{proposition}{Proposition}
\safenewtheorem{example}{Example}
\theoremstyle{definition}
\safenewtheorem{definition}{Definition}
\newtheorem{assumption}[theorem]{Assumption}
\theoremstyle{remark}
\safenewtheorem{remark}{Remark}

\newcommand{\R}{\mathbb{R}}
\newcommand{\norm}[1]{\left\lVert #1 \right\rVert}
\newcommand{\abs}[1]{\left\lvert #1 \right\rvert}

\newcommand{\dmatch}{d_{\mathrm{match}}}
\newcommand{\lmin}{\lambda_{\min}}

\newcommand{\Proj}{\Pi}
\newcommand{\Marg}{M_t}
\newcommand{\MargOver}{M_t^{\mathrm{ov}}}
\newcommand{\MargTau}{M_t^{\tau}}

\newcommand{\Sstar}{S_t^{\star}}

\newcommand{\Fov}{F_t^{\mathrm{ov}}}

\newcommand{\etaov}{\eta^{\star,\mathrm{ov}}_t}
\newcommand{\dstar}{d^{\star}_t}
\newcommand{\lstar}{\lambda^{\star}_t}
\newcommand{\dstarov}{\sigma^{\star,\mathrm{ov}}_t}
\DeclareMathOperator*{\argmax}{arg\,max}

\begin{document}
\title{Movement Stability of D-Optimal Sensor Placement\\ Under Projector Perturbations}

 \author{Isabella Yin$^{1}$ \and Laura P.\ Schaposnik$^{2,3,4}$}
 
\date{\today}

\makeatletter
\twocolumn[{%
  \begin{@twocolumnfalse}
  \maketitle
  \begin{center}\small
$^{1}$Tsinghua International School, Beijing, China\\[1pt]
$^{2}$ MSCS, LQuTE, AIRSHIP Lab, 
University of Illinois Chicago, Chicago, IL, USA\\[1pt]
$^{3}$NSF-Simons National Institute for Theory and Mathematics in
Biology, Chicago, IL, USA\\[1pt]
$^{4}$Mathematical Institute and Magdalen College,
University of Oxford, UK\\[4pt]
\texttt{bellayin08@gmail.com}\qquad
\texttt{schapos@uic.edu} (corresponding author)
\end{center}
  \blueline
  \begin{abstract}
  \noindent
In time-varying D-optimal subset selection, the number of replaced coordinates does not measure
physical displacement: when the optimal selection for a set of mobile sensors changes, a single
sensor may have to cross the entire graph. We therefore measure movement by a bottleneck matching distance on a fixed mobility graph, and define the distance margin at radius $P$ as the smallest old-objective loss among configurations farther than $P$ from the current maximizer. Our main result shows that if the projector drift is small and the distance margin at a chosen radius exceeds an explicit threshold, then every maximizer of the new objective lies within that radius, without tie-breaking rules or uniform conditioning assumptions over subsets. Moreover, we show that separation information of this kind is genuinely needed: a weighted-path construction shows that, at a fixed baseline eigengap, an arbitrarily small perturbation can move an optimal sensor from one end of the graph to the other, so no bound depending only on perturbation size and eigengap can vanish with the perturbation. We also give a finite certificate, conditions under which local and Voronoi restrictions retain a global optimum, and oversampled and regularized extensions, and we identify the learned dictionary criteria to which the results apply.
  \end{abstract}
  Bottleneck matching. 
  % your MSC codes
  \blueline
  \medskip
  \end{@twocolumnfalse}
}]
\makeatother

\section{Introduction}
\label{sec:intro}
  Where to place a few sensors so that they tell us as much as possible about a signal is an old question that has never really gone away. Classical optimal design answers it with the D-optimality criterion \cite{pukelsheim2006optimal}. By the late 2000s this had become a standard tool for sensor selection \cite{joshi2009sensor,krause2008near}, and graph signal processing has since given it a natural home on networks (e.g., see \cite{pesenson2008sampling,shuman2013emerging,chen2015discrete,anis2016efficient,ortega2018graph}). Most of this work, however, is static. In practice signal models drift and are re-estimated \cite{dilorenzo2018adaptive,nomura2024dynamic}, and while the effect of such perturbations on the underlying subspace has been understood since Davis and Kahan \cite{davis1970rotation}, much less is known about their effect on the sensors themselves. The question is not how far the subspace moves, but how far the hardware has to travel.

To see why this matters, imagine a fleet of buoys measuring sea-surface temperature along a coastline. The signal model is
updated, the D-optimal placement is recomputed, and the new answer differs from the old one by a
single sensor. On paper this is the smallest change there is. In the water, it may mean sending one
buoy from one end of the coast to the other.

This is not an artifact of messy data: indeed, it already happens in the simplest setting one can write
down. Put one sensor on a path of $n$ vertices to sample a smooth signal, with edge weights chosen
so that the two endpoints are equally good. An arbitrarily small change in those weights breaks the
tie, and the optimal sensor jumps $n-1$ hops (Theorem~\ref{thm:impossibility}). One coordinate was replaced,
yet the sensor crossed the entire graph. Counting replacements and measuring movement are
genuinely different things, and only the second is what a mobile sensor pays for.

For this reason, algorithms for dynamic placement often restrict attention to a local search region for each sensor. Nomura et al.\ update a learned dictionary, and restrict each sensor to a $P$-hop neighborhood intersected with a Voronoi cell determined by the current sensor positions \cite{nomura2024dynamic}. Their experiments on synthetic signals and sea surface temperature data show improved reconstruction error relative to certain static placement baselines. Such a local restriction is practical, but it raises the question that our paper aims to address: when is a local $P$-hop restriction guaranteed to retain a global optimum?

Answering this question requires separating four related ideas and statements:

\begin{enumerate}[label=(\alph*),leftmargin=*]

\item Set replacement: the number of replaced sensors, $\frac{1}{2}\abs{S\triangle S'}$, where $S,S'$
are the old and new sets of selected vertices and $\triangle$ denotes symmetric difference.

\item Physical movement: how far each sensor must travel, measured using a bottleneck matching distance.

\item Feasibility of a local search region or adjustment (whether a global optimum of the new objective belongs to the local region of the previous state).

\item Success of a particular algorithm: whether a greedy or distributed update finds that close optimum.

\end{enumerate}

Our paper reaches results on (b) and (c), but not (d). We make no claim about whether the independent per-sensor greedy updates in Algorithm 2 of Nomura et al.~\cite{nomura2024dynamic} solve the centralized D-optimal problem. Section~\ref{sec:bridge} gives conditions under which the local feasible region retains a global optimum; it does not establish that independent per-sensor updates recover that optimum. We discuss related work in Section~\ref{subsec:related}, and describe below our main findings:

\label{subsec:contributions}
\begin{enumerate}[label=\arabic*.,leftmargin=*]
\item {\it Movement bound (distance margin) under projector perturbation.}
We derive a stability theorem that bounds the movement of every new maximizer using projector drift, conditioning only at the current optimizer, and a distance margin, and
show that in random ensembles the probability of any optimizer change is controlled by the
lower tail of the second-best objective gap.

\item {\it Why a distance margin is needed.}
A weighted-path construction shows that no movement bound depending only on perturbation size and
eigengap can both hold uniformly over weighted graphs and vanish with the perturbation at fixed
eigengap.

\item {\it Computable certification. }
We show how to compute the distance margin exactly on finite graphs and test the movement constraint using perfect matchings.

\item {\it Local losslessness and extensions. }
We show when $P$-hop restrictions intersected with Voronoi regions preserve a global optimum,
and that without the distance margin they can lose it under arbitrarily small perturbations, and extend the stability result to oversampled and regularized settings, and identify which dictionary-based criteria the general theorem covers.
\end{enumerate}
Section~\ref{sec:setting} reviews related work and sets up the problem; Sections~\ref{sec:main}--\ref{sec:laplacian} establish the stability and impossibility results; Sections~\ref{sec:margin}--\ref{sec:extensions} present certification, local restrictions, and extensions. Section~\ref{sec:experiments} presents synthetic experiments on the certificate and the distance margin curve.

\section{Background and Setup}
\label{sec:setting}
In this section we first place the problem within the literatures it draws on, namely
graph signal sampling, D-optimal design, subspace perturbation, and dynamic sensor placement, and
identify what each leaves open for mobile sensors (Section~\ref{subsec:related}). We then fix the
setup used throughout the paper: the mobility graph on which movement is measured, the two
determinant objectives, the bottleneck matching distance, and the distance margin
(Section~\ref{subsec:setup}). Readers familiar with the literature may skip directly to the latter.

\subsection{Related work}
\label{subsec:related}
The problem studied here sits at the meeting point of four lines of work, each of which supplies
part of what we need and none of which, on its own, addresses how far mobile sensors must travel
when the signal model changes.

\begin{enumerate}[label=\arabic*.,leftmargin=*]
\item {\bf Graph signal sampling: }Sampling theory establishes which vertex subsets determine a band-limited graph signal in a low-frequency subspace \cite{anis2016efficient,chen2015discrete,marques2016sampling,pesenson2008sampling,tsitsvero2016signals} (see Shuman et al.~\cite{shuman2013emerging} and Ortega et al.~\cite{ortega2018graph} for surveys). Graphical designs provide a related quadrature viewpoint for low-frequency graph functions \cite{steinerberger2020generalized,babecki2021codes}. All of these concern a fixed signal model; none asks how an optimal set moves when the model changes.

\item {\bf{D-optimal design and subset selection:}} The D-optimal design problem in matrix sampling maximizes a log-determinant criterion, leading to the selection of informative subsets of rows, and is prevalent in problems of sensor selection \cite{joshi2009sensor,krause2008near}, maximum-volume submatrix selection \cite{civril2009selecting}, and randomized numerical linear algebra \cite{avron2013faster}. Combinatorial versions of these maximization problems are often solved using greedy and related approximation strategies. Instead of analyzing the approximation achieved by a particular algorithm, we provide theoretical guarantees on how far a maximizing configuration can move when its projector changes.

\item {\bf{Subspace and projector perturbation:}} The Davis--Kahan $\sin\Theta$ theorem \cite{davis1970rotation} and its further refinements \cite{stewart1990matrix,bhatia1997matrix,kato1995perturbation,yu2015useful} describe how an invariant subspace changes under symmetric perturbations. Classical subspace perturbation results control the change in the associated orthogonal projector. We use this control to obtain movement guarantees for determinant-maximizing configurations. Our projector-level formulation also applies to subspaces obtained from learned dictionaries.

\item {\bf{Dynamic sensor placement:}} Adaptive sampling strategies for real-time graph signal streams have been studied in \cite{dilorenzo2018adaptive}, and Nomura et al.~\cite{nomura2024dynamic} recently proposed a local-update strategy that restricts each sensor to a $P$-hop neighborhood intersected with its Voronoi cell. Such restrictions work well in practice, but nothing guarantees that the new global optimum lies inside the local region; Section~\ref{sec:bridge} gives conditions under which it does.  
The undersampled determinant objective studied here is also
one of the design-based placement criteria considered in
the TAPS framework \cite{lin2026taps}. The present work
studies the movement of its maximizers when the underlying
subspace changes.

\end{enumerate}

  Margin-based certificates also appear in \cite{gokavarapu2026finite}, where, under the
statistical assumptions of that work, a sufficiently large score margin guarantees that a
top-$m$ selection in a graph-centrality problem is unchanged under perturbation. Unlike in our problem, those objectives are coordinatewise separable, whereas the D-optimal objective is a nonseparable set function. Our margin is instead defined in terms of distance in the underlying mobility graph and is used to certify a movement radius rather than only invariance of the top-$m$ set.

\subsection{Setup}
\label{subsec:setup}
The framework has two layers, and it is worth keeping them apart from the start. The first is
geometric: a fixed graph on which the sensors move, which determines what any reconfiguration costs.
The second is the signal model: a low-dimensional subspace, encoded by an orthogonal projector,
which may change from one time step to the next. We fix the first in
Assumption~\ref{ass:movement}, then introduce the determinant objectives built from the second,
and finally combine the two through the bottleneck matching distance
(Definition~\ref{def:matching_distance}) and the distance margin (Definition~\ref{def:margin}), on
which all later results rest.
\begin{assumption}[Fixed mobility graph]
\label{ass:movement}
Let $G_{\mathrm{mv}}=(V,E_{\mathrm{mv}})$ be a fixed connected unweighted graph on the vertex set $V$ with $\abs{V}=n$, and let
$d(u,v):=\min_{\pi:u\leadsto v}\abs{\pi}$
denote its shortest-path distance, where $\pi$ ranges over paths in $G_{\mathrm{mv}}$ and $\abs{\pi}$ is the number of edges. All distances, $P$-hop neighborhoods, Voronoi cells, and matching distances in this paper refer to $G_{\mathrm{mv}}$.
\end{assumption}
The signal model, by contrast, may be a weighted graph, a spectral subspace, or a learned dictionary. In the graph-signal application of Sections~\ref{sec:laplacian} and~\ref{sec:bridge} we take $G_{\mathrm{mv}}$ to be the support graph of the weighted graph, and we assume that the support is fixed while the edge weights may change. The impossibility construction lies in this same setting, where only weights change.
\begin{remark}[What the movement metric encodes versus does not]
\label{rem:metric_scope}
The bottleneck matching distance of Definition~\ref{def:matching_distance} treats sensors as interchangeable (only asks whether the old and new positions can be matched with small displacement). It does not separate sensor types or account for trajectory collisions, travel times, or edge-capacity constraints. These limitations motivate trajectory-related extensions discussed in Section~\ref{sec:conclusion} for future work.
\end{remark}

We now turn to the signal model. Fix integers $K$ and $m$ with $1\le K<n$ and $1\le m\le n$.
For $s\in\{t,t+1\}$, let $\Proj_s\in\R^{n\times n}$ be an orthogonal projector of rank $K$, and choose $U_s\in\R^{n\times K}$ with orthonormal columns spanning its range, such that $\Proj_s=U_sU_s^\top$. Set $\rho_t:=\norm{\Proj_t-\Proj_{t+1}}_2$.

Throughout, $\norm{\cdot}_2$ is the Euclidean norm of a vector and the induced operator norm of a
matrix, $I_m$ is the $m\times m$ identity, $\lmin$ denotes the smallest eigenvalue of a
symmetric matrix, and $X\succeq Y$ ($X\succ Y$) means that $X-Y$ is positive semidefinite
(definite). A configuration is a subset $S\subseteq V$ with $\abs{S}=m$. After fixing an ordering
of $V$, let $R_S\in\R^{m\times n}$ consist of the rows of $I_n$ indexed by $S$, so that $R_Sx$
lists the entries of $x$ on $S$, and let $D_S=R_S^\top R_S$ be the associated binary diagonal
matrix. Write $U_{s,S}:=R_SU_s$ for the rows of $U_s$ indexed by $S$.

Undersampled case $m\le K$:
for $s\in\{t,t+1\}$ define
$C_s(S):=R_S\Proj_sR_S^\top\in\R^{m\times m},$ and $F_s(S):=\log\det C_s(S),$ with $F_s(S)=-\infty$ when $C_s(S)$ is singular. Maximizing $F_s$ selects rows of the subspace basis according to the D-optimal volume criterion \cite{pukelsheim2006optimal,chen2015discrete}. For $m<K$, this is a row-volume criterion; when $m=K$, the compressed-projector and information-matrix determinants agree.

Oversampled case $m\ge K$:
for $m>K$, the matrix $C_s(S)$ has rank at most $K<m$, so its determinant is zero and its log-determinant is $-\infty$. Thus, for $s\in\{t,t+1\}$ define instead $H_s(S):=U_s^\top D_SU_s\in\R^{K\times K},$ and $F_s^{\mathrm{ov}}(S):=\log\det H_s(S),$ again with value $-\infty$ when $H_s(S)$ is singular.

At $m=K$ when the two determinant objectives agree, the undersampled perturbation estimate will be sharper and we use it in that boundary case.

\begin{definition}[Maximizers]
\label{def:optimizers}
Write $\argmax F_t$ for the set of configurations maximizing $F_t$ over all $m$-subsets; it is nonempty because there is a finite possible family. A reference maximizer $\Sstar\in\argmax F_t$ is any fixed element of this set, relative to which all margins are defined. No tie-breaking is needed at time $t+1$; the results hold for every $S\in\argmax F_{t+1}$.
\end{definition}
\begin{lemma}[Nondegeneracy in undersampled case]
\label{lem:nondegenerate}
If $m\le K$ then some $m$-subset $S$ has $\det C_t(S)>0$; consequently $\dstar:=\det C_t(\Sstar)=e^{F_t(\Sstar)}>0$ for every reference maximizer.
\end{lemma}
\begin{proof}Since $U_t$ has rank $K\ge m$,   it contains $m$ linearly independent rows, indexed by some $S$. Then $U_{t,S}\in\R^{m\times K}$ has rank $m$ and $C_t(S)=U_{t,S}U_{t,S}^\top$ is positive definite. Any maximizer thus has objective value at least $F_t(S)>-\infty$
\end{proof} 

Thus, in the undersampled regime, the nondegeneracy used below is automatic for every rank-$K$ projector; no genericity assumption on the graph weights or projector entries is required.\smallskip

\begin{definition}[Bottleneck matching distance]
\label{def:matching_distance}
For $m$-sets $S=\{s_1,\dots,s_m\}$ and $S'=\{s_1',\dots,s_m'\}$, put
\begin{equation}
\begin{split}
    \dmatch(S,S')&:=\min_{\pi:\{1,\ldots,m\}\to\{1,\ldots,m\}\ \mathrm{bijective}}\\
    &\qquad\max_{1\le i\le m} d\bigl(s_i,s'_{\pi(i)}\bigr)
\end{split}
    \label{eq:dmatch_def}
\end{equation}
with $d$ as the distance of the fixed movement graph $G_{\mathrm{mv}}$. This is the smallest radius within which every old sensor can be matched bijectively to a new sensor. The value is independent of the ordering chosen for $S$ and $S'$.
\end{definition}
\begin{proposition}[Metric property]
\label{prop:matching_metric}
$\dmatch$ is a metric on the family of $m$-subsets of $V$.
\end{proposition}
\begin{proof}The minimum in \eqref{eq:dmatch_def} is attained because there are only finitely many bijections between two $m$-element sets. Nonnegativity and symmetry thus follow directly from the corresponding properties of $d$. If $\dmatch(S,S')=0$, then some bijection pairs each $s_i$ with a vertex at graph distance zero (hence with $s_i$ itself) so $S=S'$. Conversely, $\dmatch(S,S)=0$. For the triangle inequality, let $S=\{s_i\}$, $T=\{t_j\}$, and $R=\{r_k\}$, and choose bijections $\pi$ and $\sigma$ attaining $\dmatch(S,T)$ and $\dmatch(T,R)$, respectively. Composing gives us
\begin{eqnarray*}
  d\bigl(s_i,r_{\sigma(\pi(i))}\bigr) 
  &\le& d\bigl(s_i,t_{\pi(i)}\bigr)+d\bigl(t_{\pi(i)},r_{\sigma(\pi(i))}\bigr) \\
  &\le& \dmatch(S,T)+\dmatch(T,R)
\end{eqnarray*}
for every $i$. Taking the maximum over $i$ and then the minimum over bijections from $S$ to $R$ gives us the claim.
\end{proof}
\begin{definition}[Distance margin]
\label{def:margin}
Fix a reference maximizer $\Sstar\in\argmax F_t$. For an integer $P\ge0$ define
\begin{equation}
  \Marg(P):=\min_{\substack{\abs{S}=m\\ \dmatch(S,\Sstar)>P}}\bigl[F_t(\Sstar)-F_t(S)\bigr],
  \label{eq:margin_def}
\end{equation}
with convention $\Marg(P)=+\infty$ when the minimum is over an empty family. Equivalently, among all configurations that cannot be reached from $S^\star_t$ by moving each sensor at most $P$ hops, $M_t(P)$ measures how much worse the best one is under the old objective.
\end{definition}

The same empty-family convention applies to $\MargOver$ and $\MargTau$ introduced later, so the corresponding stability statements hold when their feasible families are empty.

The stability result depends on three main quantities. Projector drift measures the perturbation of the underlying low-dimensional subspace, the distance margin measures the gap between the optimal configuration and distant competitors, and perturbation budget measures the objective change caused by projector drift. A summary of the notation used throughout is given in Appendix~\ref{app:notation} in Tables~\ref{tab:notation}–\ref{tab:notation3}.

\section{Movement Stability Under Projector Perturbations}
\label{sec:main}
Throughout this section $m\le K$, and $\Proj_t,\Proj_{t+1}$ are arbitrary rank-$K$ orthogonal projectors. The argument does not require the projectors to come from a graph; the fixed mobility graph comes in only through the matching distance of Definition~\ref{def:matching_distance}. 
The proof has four steps. Lemma~\ref{lem:compression} transfers the projector drift to every
compressed matrix, and Lemma~\ref{lem:logdet} turns this into a bound on the change of the
objective at the reference maximizer, which is well conditioned. Optimality at time $t+1$ then
forces every new maximizer to have a large determinant and hence, by
Proposition~\ref{prop:auto}, to be well conditioned too. This allows each new maximizer to be
compared back to the old objective, bounding its old-objective loss by $B_t$. Finally, if the
distance margin at radius $P$ exceeds $B_t$, no configuration farther than $P$ from the reference
can have so small a loss. The smallness conditions that make each step valid are stated in
Theorem~\ref{thm:main}.
 
\begin{lemma}[Compression]
\label{lem:compression}
For every $m$-subset $S$,
\[\norm{C_t(S)-C_{t+1}(S)}_2\le\rho_t.\]
\end{lemma}
\begin{proof}$C_t(S)-C_{t+1}(S)=R_S(\Proj_t-\Proj_{t+1})R_S^\top$. Since $R_S$ has orthonormal rows, $\norm{R_S}_2=\norm{R_S^\top}_2=1$, and submultiplicativity gives the bound.
\end{proof}
\begin{lemma}[Log-determinant stability]
\label{lem:logdet}
Let $X,Y\in\R^{m\times m}$ be symmetric with $X\succ0$ and $c:=\lmin(X)>0$. If
\begin{equation}
  \norm{Y-X}_2\le r\le\frac{c}{2},
  \label{eq:logdet_smallness}
\end{equation}
then $Y\succ0$ and
\begin{equation}
  \abs{\log\det Y-\log\det X}\le\frac{2mr}{c}.
  \label{eq:logdet_bound}
\end{equation}
\end{lemma}
\begin{proof}Write $E:=Y-X$. By Weyl's inequality \cite[Cor.~4.3.15]{hornjohnson2013}, $\lmin(Y)\ge c-r\ge c/2>0$, so $Y\succ0$. Since $X\succ0$ it has a symmetric positive-definite square root, and we may set  $ B:=X^{-1/2}EX^{-1/2}$ so that 
$  Y=X+E=X^{1/2}(I+B)X^{1/2}.$
 The matrix $B$ is symmetric because $E$ and $X^{-1/2}$ are, hence its eigenvalues $\nu_1,\dots,\nu_m$ are real. Taking determinants,
\begin{eqnarray*}
  \log\det Y-\log\det X&=&\log\det(I+B)\\
  &=&\sum_{i=1}^m\log(1+\nu_i).
\end{eqnarray*}
By submultiplicativity, $\norm{B}_2\le\norm{X^{-1/2}}_2^2\norm{E}_2=\norm{E}_2/c\le r/c\le1/2$, so $\abs{\nu_i}\le1/2$ for every $i$. Then, for $\abs{x}\le\tfrac12$ the scalar inequality
\begin{equation}
\begin{gathered}
  \abs{\log(1+x)}\le2\abs{x}
\end{gathered}
\label{eq:scalar_log}
\end{equation}
holds because for $x\in[0,1/2]$, one has $\log(1+x)\le x\le 2\abs{x}$, while for $x\in[-1/2,0]$,
\begin{eqnarray*}
  \abs{\log(1+x)}&=&-\log(1+x)\\
  &=&\int_{1+x}^{1}\frac{dt}{t}\le 2(-x)\\
  &=&2\abs{x},
\end{eqnarray*}
since $t\ge1/2$ throughout the interval of integration. Therefore
\begin{eqnarray*}
  \abs{\log\det Y-\log\det X} &\le& \sum_{i=1}^m\abs{\log(1+\nu_i)}\\
  &\le&2\sum_{i=1}^m\abs{\nu_i}\\
  &\le&2m\norm{B}_2\\
  &\le&\frac{2mr}{c}.
\end{eqnarray*}

This proves \eqref{eq:logdet_bound}.\end{proof}

Lemma~\ref{lem:logdet} gives the same conclusion if $X$ and $Y$ are exchanged, but its hypothesis must be checked for $X$. Thus, we may apply it in either direction between $C_t(S)$ and $C_{t+1}(S)$ if the minimum eigenvalue condition holds for chosen $X$. We use both directions below.
\begin{proposition}[Determinant lower bound for the smallest eigenvalue]
\label{prop:auto}
Let $\Proj$ be an orthogonal projector, $S$ an $m$-subset, and $C=(\Proj)_{SS}$. Then $0\preceq C\preceq I_m$, and if $\det C>0$ then
\begin{equation}
  \lmin(C)\ge\det C.
  \label{eq:auto}
\end{equation}
\end{proposition}
\begin{proof}Since $0\preceq\Proj\preceq I_n$, congruence by $R_S$ gives $0\preceq C\preceq I_m$, so every eigenvalue of $C$ lies in $[0,1]$. Writing $\det C=\prod_{i=1}^m\lambda_i(C)$ and bounding every factor other than the smallest by $1$ gives us \eqref{eq:auto}. 
\end{proof}

Proposition~\ref{prop:auto} shows that a configuration with determinant bounded away from zero automatically has its smallest eigenvalue bounded away from zero as well (and hence has controlled condition number because $\lambda_{\max} \leq 1$);  the conditioning step in Theorem~\ref{thm:main} supplies this positivity. Thus Theorem~\ref{thm:main} requires no uniform lower bound on $\lmin(C_t(S))$ over all configurations. This argument applies only when $m\le K$; the oversampled case uses the separate conditioning argument that will be introduced later in Section~\ref{subsec:oversampled}.

\begin{theorem}[Distance margin stability under projector perturbation]
\label{thm:main}
Assume $m\le K$ and Assumption~\ref{ass:movement}. Let $\Sstar\in\argmax F_t$ be a reference maximizer and define
 $ \lstar:=\lmin\bigl(C_t(\Sstar)\bigr),$ and
  $\dstar:=\det C_t(\Sstar)>0 .$ 
Let $r$ satisfy $\rho_t\le r$, and define
\begin{gather*}
\eta_{t,1}:=\frac{2mr}{\lstar},\text{ and }
\eta_{t,2}:=\frac{2mr\,e^{\eta_{t,1}}}{\dstar},\\\text{ and }
B_t:=\eta_{t,1}+\eta_{t,2}.
\end{gather*}
Suppose the two smallness conditions

\begin{equation}
\begin{gathered}
r\le\frac{\lstar}{2},
\qquad
r\le\frac{\dstar e^{-\eta_{t,1}}}{2}.
\end{gathered}
\label{eq:smallness}
\end{equation}

hold. Then:
\begin{enumerate}[label=(\alph*),leftmargin=*]
\item Objective loss. For every $S\in\argmax F_{t+1}$,
\begin{equation}
  F_t(\Sstar)-F_t(S)\le B_t.
  \label{eq:objective_loss}
\end{equation}
\item Movement. For every integer $P\ge0$,
\begin{equation}
\begin{gathered}
  \Marg(P)>B_t
  \Longrightarrow\quad \dmatch(\Sstar,S)\le P\\
  \text{ for every }S\in\argmax F_{t+1}.
\end{gathered}
\label{eq:main_implication}
\end{equation}
\end{enumerate}
\end{theorem}
\begin{proof}By Lemma~\ref{lem:nondegenerate}, $\dstar>0$, and by Proposition~\ref{prop:auto}, $\lstar\ge\dstar>0$.

First, we control the drift at the reference maximizer. By Lemma~\ref{lem:compression}, $\norm{C_{t+1}(\Sstar)-C_t(\Sstar)}_2\le\rho_t\le r$, and by the first condition in \eqref{eq:smallness}, $r\le\lstar/2$. Lemma~\ref{lem:logdet} with $X=C_t(\Sstar)$, $Y=C_{t+1}(\Sstar)$, and $c=\lstar$ gives $C_{t+1}(\Sstar)\succ0$ and
\begin{equation}
  \abs{F_{t+1}(\Sstar)-F_t(\Sstar)}\le\frac{2mr}{\lstar}=\eta_{t,1}.
  \label{eq:step1}
\end{equation}
Next, let $S\in\argmax F_{t+1}$. By optimality at time $t+1$ and \eqref{eq:step1},
\[
  F_{t+1}(S)\ge F_{t+1}(\Sstar)\ge F_t(\Sstar)-\eta_{t,1}=\log\dstar-\eta_{t,1},
\]
and so $\det C_{t+1}(S)\ge\dstar e^{-\eta_{t,1}}>0$. The matrix $C_{t+1}(S)$ is positive semidefinite with positive determinant, hence positive definite, and Proposition~\ref{prop:auto} applied to $\Proj_{t+1}$ gives us
\begin{equation}
  \lmin\bigl(C_{t+1}(S)\bigr)\ge\det C_{t+1}(S)\ge\dstar e^{-\eta_{t,1}}.
  \label{eq:step2}
\end{equation}
Now compare the new maximizer back to time $t$. By Lemma~\ref{lem:compression}, $\norm{C_t(S)-C_{t+1}(S)}_2\le\rho_t\le r$, and the second condition in \eqref{eq:smallness} together with \eqref{eq:step2} gives
\[
  r\le\frac{\dstar e^{-\eta_{t,1}}}{2}\le\frac{\lmin\bigl(C_{t+1}(S)\bigr)}{2}.
\]
Lemma~\ref{lem:logdet}, applied with $X=C_{t+1}(S)$ and $Y=C_t(S)$, therefore gives
\begin{eqnarray}
  \abs{F_t(S)-F_{t+1}(S)}
  &\le&\frac{2mr}{\lmin(C_{t+1}(S))}\nonumber\\
  &\le&\frac{2mr\,e^{\eta_{t,1}}}{\dstar}
  =\eta_{t,2}.
  \label{eq:step3}
\end{eqnarray}
Combining \eqref{eq:step1}, optimality of $S$ for $F_{t+1}$, and \eqref{eq:step3}:
\begin{align*}
&F_t(\Sstar)-F_t(S)
  =\underbrace{\bigl[F_t(\Sstar)-F_{t+1}(\Sstar)\bigr]}_{\le\;\eta_{t,1}}\\
&\qquad \qquad \qquad \quad+\underbrace{\bigl[F_{t+1}(\Sstar)-F_{t+1}(S)\bigr]}_{\le\;0}\\
&\qquad \qquad \qquad \quad+\underbrace{\bigl[F_{t+1}(S)-F_t(S)\bigr]}_{\le\;\eta_{t,2}} \le B_t,
\end{align*}
which is \eqref{eq:objective_loss}.

Finally, assume $\Marg(P)>B_t$, let $S\in\argmax F_{t+1}$, and suppose for the sake of contradiction that $\dmatch(\Sstar,S)>P$. Then $S$ is feasible in the minimum defining $\Marg(P)$, so $F_t(\Sstar)-F_t(S)\ge\Marg(P)>B_t$, contradicting \eqref{eq:objective_loss}. Therefore, $\dmatch(\Sstar,S)\le P$ which proves the statement.
\end{proof}

\begin{corollary}[Determinant-based movement bound]
\label{cor:simple}
Under the hypotheses of Theorem~\ref{thm:main} except \eqref{eq:smallness}, suppose instead that
\begin{equation}
\begin{gathered}
  r\le\frac{\dstar}{4}
\quad\text{and}\quad
\frac{2mr}{\dstar}\le\log2.
\end{gathered}
\label{eq:simple_smallness}
\end{equation}
Then \eqref{eq:smallness} holds, $B_t\le\frac{6mr}{\dstar}$, and so
\begin{equation}
\begin{gathered}
  \Marg(P)>\frac{6mr}{\dstar}~
  \Longrightarrow\quad \dmatch(\Sstar,S)\le P~\\
  \text{for every }S\in\argmax F_{t+1}.
\end{gathered}
\label{eq:simple_implication}
\end{equation}
\end{corollary}
\begin{proof}By Proposition~\ref{prop:auto}, $\lstar\ge\dstar$, so $r\le\dstar/4\le\lstar/2$, which is the first condition in \eqref{eq:smallness}, and $\eta_{t,1}=2mr/\lstar\le2mr/\dstar\le\log2$. Hence $e^{-\eta_{t,1}}\ge\tfrac12$ and
\[
  \frac{\dstar e^{-\eta_{t,1}}}{2}\ge\frac{\dstar}{4}\ge r,
\]
which is the second condition in \eqref{eq:smallness}. Finally $e^{\eta_{t,1}}\le e^{\log2}=2$, so $\eta_{t,2}\le2\cdot2mr/\dstar=4mr/\dstar$ and $B_t\le2mr/\dstar+4mr/\dstar=6mr/\dstar$. Thus $\Marg(P)>6mr/\dstar$ implies $\Marg(P)>B_t$, and the second part of Theorem~\ref{thm:main} gives the conclusion. 
\end{proof}

We have not optimized the constant $6m$ and make no claim of sharpness in $m$.

\begin{corollary}[Average-case persistence from the second-best gap]
\label{cor:average_case}
Assume the hypotheses of Theorem~\ref{thm:main}, and suppose that $\Sstar$ is the unique maximizer of $F_t$. Define its second-best gap by
\[
g_t
:=
F_t(\Sstar)
-
\max_{\substack{\abs S=m\\ S\ne\Sstar}}F_t(S)
>0.
\]
If $F_{t+1}$ has a maximizer different from $\Sstar$, then
\[
g_t\le B_t.
\]

Consequently, consider a family of random baseline--perturbed
pairs indexed by a perturbation scale $\varepsilon>0$.
Assume that the hypotheses of Theorem~3.4 hold almost surely
and that the baseline maximizer $S_t^\star$ is unique almost
surely. Suppose there exist constants $C,c,L,x_0>0$,
independent of $\varepsilon$, such that
\[
  \Pr(g_t\le x)\le Cx^c
  \qquad\text{for }0\le x\le x_0,
\]
and
\[
  B_t\le L\varepsilon
  \qquad\text{almost surely}.
\]
Then, for $0<\varepsilon\le x_0/L$,
\[
  \Pr\bigl(\argmax F_{t+1}\ne\{S_t^\star\}\bigr)
  \le CL^c\varepsilon^c.
\]
\end{corollary}

\begin{proof}
Suppose that $F_{t+1}$ has a maximizer $S\ne\Sstar$. By Theorem~\ref{thm:main}(a),
\[
F_t(\Sstar)-F_t(S)\le B_t.
\]
Since $g_t$ is the smallest old objective loss among configurations different from $\Sstar$,
\[
g_t
\le
F_t(\Sstar)-F_t(S)
\le
B_t.
\]
Hence the event that the optimizer changes is contained in the event $\{g_t\le B_t\}$. Under the assumed bound $B_t\le L\varepsilon$,
\[
\begin{aligned}
\Pr\!\left(\argmax F_{t+1}\ne\{\Sstar\}\right)
&\le \Pr(g_t\le L\varepsilon)\\
&\le C(L\varepsilon)^c.
\end{aligned}
\]
which is the claim. \end{proof}

For Laplacian perturbations, choose
$r=4\varepsilon/\delta$ as supplied by Lemma~4.2.
On any family with fixed $m$, a uniformly positive eigengap,
and uniformly positive lower bounds on $\lambda_t^\star$
and $d_t^\star$, the smallness conditions hold for
sufficiently small $\varepsilon$, and $B_t=O(\varepsilon)$
uniformly over the family. Determining the lower-tail
exponent $c$ for natural random graph or data-driven graph
models remains an open problem.

\section{Laplacian Perturbations and Why Projector Drift Alone Is Not Enough}
\label{sec:laplacian}
We now consider the case where $\Proj_t$ is the low-frequency spectral projector of a graph Laplacian, and show that perturbation size and eigengap alone cannot provide a movement bound that tends to zero with the perturbation.

Let $G_t=(V,E,w_t)$ be a connected weighted undirected graph with positive edge weights $w_{t,ij}$, and write $j\sim i$
when $\{i,j\}\in E$. Let $A_t$ be the weighted adjacency matrix, $D_t$ the diagonal matrix of
weighted degrees, and $L_t=D_t-A_t$ the combinatorial Laplacian, so that $(L_t\mathbf x)_i=\sum_{j\sim i}w_{t,ij}(x_i-x_j)$. Its eigenvalues are $0=\lambda_1\le\cdots\le\lambda_n$ with orthonormal eigenvectors $\phi_1,\dots,\phi_n$, the graph Fourier modes \cite{chung1997spectral,shuman2013emerging}. For a bandwidth $K<n$ set $U_t=[\phi_1,\dots,\phi_K]$ and $\Proj_t=U_tU_t^\top$.
\begin{assumption}[Laplacian perturbation]
\label{ass:spectral}
The edge set $E$ is the same at times $t$ and $t+1$, only the weights change, and $G_{\mathrm{mv}}$ is the unweighted graph $(V,E)$. Furthermore,
  $ \varepsilon:=\norm{L_t-L_{t+1}}_2,$ and 
 $ \delta:=\lambda_{K+1}(L_t)-\lambda_K(L_t)>0,$ and
 $ \varepsilon<\frac{\delta}{4}.$
 \end{assumption}
\begin{lemma}[Projector drift via Davis--Kahan]
\label{lem:dk}
Under Assumption~\ref{ass:spectral}, the rank-$K$ spectral projector $\Proj_{t+1}$ is uniquely determined and
\begin{equation}
  \rho_t=\norm{\Proj_t-\Proj_{t+1}}_2\le\frac{\varepsilon}{\delta-\varepsilon}\le\frac{4\varepsilon}{3\delta}\leq\frac{4\varepsilon}{\delta}.
  \label{eq:dk_bound}
\end{equation}
\end{lemma}
\begin{proof} Write $L_{t+1}=L_t+\Delta$ with $\norm{\Delta}_2\le\varepsilon$. By Weyl's inequality \cite[Cor.~4.3.15]{hornjohnson2013},
\[
  \lambda_{K+1}(L_{t+1})-\lambda_K(L_{t+1})\ge\delta-2\varepsilon>0,
\]
so the rank-$K$ spectral projector of $L_{t+1}$ is well defined. Let $U_t^\perp$ consist of eigenvectors of $L_t$ for $\lambda_{K+1}(L_t),\dots,\lambda_n(L_t)$, with $(U_t^\perp)^\top L_tU_t^\perp=\Lambda_t^\perp$, and let $U_{t+1}$ contain the first $K$ eigenvectors of $L_{t+1}$, with $L_{t+1}U_{t+1}=U_{t+1}\Lambda_{t+1}$. Put $X:=(U_t^\perp)^\top U_{t+1}$. Then
\begin{eqnarray*}
  \Lambda_t^\perp X-X\Lambda_{t+1}
  &=&(U_t^\perp)^\top L_tU_{t+1}\\
  &&-(U_t^\perp)^\top L_{t+1}U_{t+1}\\
  &=&-(U_t^\perp)^\top \Delta U_{t+1}.
\end{eqnarray*}
Every eigenvalue of $\Lambda_t^\perp$ is at least $\lambda_{K+1}(L_t)$ and every eigenvalue of $\Lambda_{t+1}$ is at most $\lambda_K(L_{t+1})\le\lambda_K(L_t)+\varepsilon$, so the two groups of eigenvalues are separated by at least $\delta-\varepsilon>0$. The Davis--Kahan $\sin\Theta$ theorem \cite{davis1970rotation,stewart1990matrix} gives
\[
  \norm{X}_2\le\frac{\norm{(U_t^\perp)^\top \Delta U_{t+1}}}{\delta-\varepsilon}\le\frac{\varepsilon}{\delta-\varepsilon}.
\]
For orthogonal projectors of equal rank, $\norm{\Proj_t-\Proj_{t+1}}_2=\norm{(I-\Proj_t)\Proj_{t+1}}_2$ \cite[Thm.~I.5.5]{stewart1990matrix}, and
\begin{eqnarray*}
  (I-\Proj_t)\Proj_{t+1}&=&U_t^\perp\bigl[(U_t^\perp)^\top U_{t+1}\bigr]U_{t+1}^\top\\
  &=&U_t^\perp XU_{t+1}^\top .
\end{eqnarray*}
Multiplication by matrices with orthonormal columns preserves the nonzero singular values, so $\norm{\Proj_t-\Proj_{t+1}}_2=\norm{X}_2$. And so finally, $\varepsilon<\delta/4$ gives $\delta-\varepsilon>3\delta/4$, which gives the last two inequalities in \eqref{eq:dk_bound}. 
\end{proof}

\begin{corollary}[Laplacian projector movement stability]
\label{cor:laplacian}
Assume $m\le K$ and Assumption~\ref{ass:spectral}, and let $\Sstar\in\argmax F_t$ with $\dstar=\det C_t(\Sstar)$. If
\begin{equation}
\begin{gathered}
  \frac{4\varepsilon}{\delta}\le\frac{\dstar}{4}~
  \text{and}~
  \frac{8m\varepsilon}{\dstar\delta}\le\log2,
\end{gathered}
\label{eq:laplacian_smallness}
\end{equation}
then for every integer $P\ge0$,
\begin{equation}
\begin{gathered}
  \Marg(P)>\frac{24m\varepsilon}{\dstar\delta}~
  \Longrightarrow\quad \dmatch(\Sstar,S)\le P~\\
  \text{for every }S\in\argmax F_{t+1}.
\end{gathered}
\label{eq:laplacian_implication}
\end{equation}
\end{corollary}
\begin{proof}Use Corollary~\ref{cor:simple} with $r=4\varepsilon/\delta$, which satisfies $r\ge\rho_t$ by Lemma~\ref{lem:dk}. Conditions \eqref{eq:laplacian_smallness} are  exactly \eqref{eq:simple_smallness} for this $r$, since $2mr/\dstar=8m\varepsilon/(\dstar\delta)$, and the threshold is $6mr/\dstar=24m\varepsilon/(\dstar\delta)$. 
\end{proof}

Corollary~\ref{cor:laplacian} uses the simplified bound $4\varepsilon/\delta$ to keep the constants uniform. Using $r=\varepsilon/(\delta-\varepsilon)$ instead gives a smaller threshold $6m\varepsilon/\bigl(\dstar(\delta-\varepsilon)\bigr)$ under the smallness conditions, and using $r=\rho_t$ directly, when $\rho_t$ is computed, is still better.
\subsection{Perturbation Size and Eigengap Alone Do Not Bound Movement}
We now show that perturbation size and eigengap alone cannot control movement as the perturbation tends to zero. Our construction will keep the graph structure fixed and all edge weights positive, while an arbitrarily small weight change moves the optimizer across the entire path.
\begin{lemma}[Discrete Sturm oscillation on a weighted path]
\label{lem:sturm}
Let $L$ be the Laplacian of a path on $n$ vertices with strictly positive edge weights. Then the eigenvalues of $L$ are simple,
 $ 0=\lambda_1<\lambda_2<\cdots<\lambda_n,$
and, after deleting the zero coordinates,
an eigenvector for $\lambda_j$ has exactly $j-1$ sign changes. In particular, a nonconstant eigenvector with exactly one sign change belongs to $\lambda_2$.
\end{lemma}
\begin{proof}A positively weighted path Laplacian is a real symmetric tridiagonal (Jacobi) matrix with nonzero off-diagonal entries $-w_i<0$. For any $\lambda$, deleting the first row and last column of $L-\lambda I$ leaves an
upper-triangular matrix with diagonal entries $-w_1,\dots,-w_{n-1}\ne0$, so $L-\lambda I$ has
rank at least $n-1$; every eigenvalue therefore has a one-dimensional eigenspace, and since $L$
is symmetric, its eigenvalues are simple. Fiedler's oscillation theorem \cite[Thm.~2.3]{fiedler1975acyclic} states that, after deleting zero coordinates, an eigenvector for $\lambda_j$ has exactly $j-1$ sign changes; see also \cite{fiedler1975property}. Therefore an eigenvector with exactly one sign change corresponds to $\lambda_2$.
\end{proof}

\begin{theorem}[No Movement Bound With Only Perturbation Size and Eigengap]
\label{thm:impossibility}
There is no function $g:(0,\infty)^2\to\R_{\ge0}$ satisfying both:
\begin{enumerate}[label=(\alph*),leftmargin=*]
\item for every fixed $\delta>0$, $\displaystyle\lim_{\varepsilon\to 0^+}g(\varepsilon,\delta)=0$.
\item for every $\varepsilon>0$ and every pair of connected weighted-graph Laplacians $L_t,L_{t+1}$ on a common vertex set and common edge set with $\norm{L_t-L_{t+1}}_2\le\varepsilon$ and $\delta_t:=\lambda_3(L_t)-\lambda_2(L_t)>0$, every reference maximizer $\Sstar$ and every $S\in\argmax F_{t+1}$ of the one-sensor, $K=2$ objective satisfy $\dmatch(\Sstar,S)\le g(\varepsilon,\delta_t)$.
\end{enumerate}
In more precise words, for every $n\ge4$, there is a positively weighted path $L_t$ on $n$ vertices with fixed baseline eigengap $\delta_n>0$ and a sequence of positively weighted path Laplacians $L^{(k)}$ on the same edge set such that $\varepsilon_k:=\norm{L^{(k)}-L_t}_2\to 0^+$, while the optimizer moves a matching distance of $n-1$. The same holds under any fixed tie-breaking rule at time $t$.
\end{theorem}

\begin{proof} Assume for the sake of contradiction that there exists a function
$g:(0,\infty)^2\to\mathbb{R}_{\ge0}$ satisfying both (a) and (b).
We construct a sequence of perturbations that violates this property.

Set $n\ge4$. For $1\le i\le n$, define
\[
a_i:=n+1-2i,
\qquad
q_i:=\Bigl(i-\frac{n+1}{2}\Bigr)^2-\frac{n^2-1}{12}.
\]
Then set $u_i(\theta):=a_i+\theta q_i.$ Both $a$ and $q$ have zero sum, the first by antisymmetry $a_{n+1-i}=-a_i$ and the second because $\sum_{i=1}^n\bigl(i-\frac{n+1}2\bigr)^2=\frac{n(n^2-1)}{12}$. Hence $u(\theta)\perp\mathbf1$ for every $\theta$. Write $A_i(\theta):=\sum_{j\le i}u_j(\theta)$
for the partial sums. At $\theta=0$,
$a_i-a_{i+1}=2\;$ for $1\le i\le n-1,$ and $A_i(0)=i(n-i)>0$ for $1\le i\le n-1,$ and 
\begin{gather*}
  a_1=n-1>0>1-n=a_n,\\
  \abs{a_1}=\abs{a_n}=n-1>\max_{2\le i\le n-1}\abs{a_i}=n-3 .
\end{gather*}
$u_i(\theta)$, their consecutive differences, and the partial sums $A_i(\theta)$ depend continuously on $\theta$. Therefore, all strict inequalities that hold at $\theta=0$ also hold for sufficiently small $\abs{\theta}$. And thus, there exists $\theta_0>0$ such that, whenever $\abs{\theta}<\theta_0$:
\begin{enumerate}[label=(\roman*)]
\item $u_1(\theta)>u_2(\theta)>\cdots>u_n(\theta)$;
\item $A_i(\theta)>0$ for $1\le i\le n-1$;
\item $u_1(\theta)>0>u_n(\theta)$;
\item $\max_i\abs{u_i(\theta)}$ is attained only at $i\in\{1,n\}$;
\item $u_1(\theta)$ and $u_n(\theta)$ have the same signs as $a_1$ and $a_n$.
\end{enumerate}
We can take $\theta_0$ to be the minimum of the finitely many positive radii by these continuity arguments.

For $\abs\theta<\theta_0$, define the path edge weights
\[
w_i(\theta):=
\frac{A_i(\theta)}{u_i(\theta)-u_{i+1}(\theta)},
\text{  for } 1\le i\le n-1.
\]
By (i) and (ii), the numerator and denominator are positive, so $w_i(\theta)>0$. Let $L(\theta)$ be the corresponding weighted-path Laplacian. Since $w_i(\theta)\bigl(u_i(\theta)-u_{i+1}(\theta)\bigr)=A_i(\theta)$,
\begin{align*}
  \bigl(L(\theta)u(\theta)\bigr)_1&=u_1(\theta),\\
  \bigl(L(\theta)u(\theta)\bigr)_i&=u_i(\theta)\qquad (2\le i\le n-1),\\
  \bigl(L(\theta)u(\theta)\bigr)_n&=u_n(\theta),
\end{align*}
where the last equality uses $\sum_{j=1}^n u_j(\theta)=0$. Thus $L(\theta)u(\theta)=u(\theta)$, so $u(\theta)$ is an eigenvector with eigenvalue $1$. By (i) and (iii), after deleting a possible zero coordinate, $u(\theta)$ has exactly one sign change. Lemma~\ref{lem:sturm} therefore gives $\lambda_2(L(\theta))=1$ for all $\abs{\theta}<\theta_0$. Since the eigenvalues are simple, $\delta_n:=\lambda_3\bigl(L(0)\bigr)-1>0.$

For $K=2$ the low-frequency projector of $L(\theta)$ is
\begin{gather*}
\Proj(\theta)=\frac1n\mathbf1\,\mathbf1^\top+\widehat u(\theta)\widehat u(\theta)^\top,~ {\rm and ~}
  \widehat u(\theta)=\frac{u(\theta)}{\norm{u(\theta)}_2},
\end{gather*}
because $\lambda_1=0$ has eigenvector $\mathbf1/\sqrt n$ and $\lambda_2=1$ has eigenvector $\widehat u(\theta)$. With one sensor, $m=1\le K$ and
\[
  F(\{i\})=\log\Proj(\theta)_{ii}=\log\Bigl(\tfrac1n+\widehat u_i(\theta)^2\Bigr),
\]
so maximizing $F$ over single vertices is equivalent to maximizing $\abs{u_i(\theta)}$.

At $\theta=0$, the maximizers are $\{1\}$ and $\{n\}$, by $\abs{a_1}=\abs{a_n}>\abs{a_i}$ for interior $i$. For $\theta\ne0$ small, property (v) gives $\abs{u_1(\theta)}=a_1+\theta q_1$ and $\abs{u_n(\theta)}=a_1-\theta q_1$. This then gives
\begin{gather*}
  \abs{u_1(\theta)}-\abs{u_n(\theta)}=2\theta q_1,\\
  q_1=q_n=\frac{(n-1)(n-2)}{6}>0 .
\end{gather*}
By (iv) no interior vertex can be maximal, so the unique maximizer is $\{1\}$ when $\theta>0$ and $\{n\}$ when $\theta<0$. Set $L_t:=L(0)$ (whose baseline eigengap is $\delta_t=\delta_n$). For sufficiently small $\theta\ne0$, $L(\theta)\ne L(0)$: otherwise $u(0)$ and $u(\theta)$ would both lie in the eigenspace of $L(0)$ for the simple eigenvalue $1$, and thus be proportional, forcing $q$ to be a multiple of $a$.

This is impossible because $q\ne0$ is symmetric, $q_{n+1-i}=q_i$, whereas $a\ne0$ is antisymmetric, $a_{n+1-i}=-a_i$. Since each $w_i(\theta)$ is a rational function with nonvanishing denominator on $(-\theta_0,\theta_0)$, $L(\theta)\to L(0)$ in operator norm as $\theta\to0$. Therefore we may choose a sequence $\theta_k<0$ tending to zero such that the positive numbers
\[
  \varepsilon_k:=\norm{L(\theta_k)-L(0)}_2
\]
are strictly decreasing. Then $\varepsilon_k\to 0^+$. Put $L^{(k)}:=L(\theta_k)$. For every $k$, the perturbed objective has the unique maximizer $\{n\}$. Choose the reference maximizer $\Sstar=\{1\}$, since $\{1\}\in\argmax F_t$. Then
\begin{gather*}
  \dmatch(\Sstar,S)=d(1,n)=n-1\\
  \text{ for the unique }S\in\argmax F_{t+1},
\end{gather*}
because $G_{\mathrm{mv}}$ is the unweighted path. Property (b) of a hypothetical $g$ would give $n-1\le g(\varepsilon_k,\delta_n)$ for every $k$, contradicting (a) as $k\to\infty$.

Finally, if a fixed rule selects one maximizer at time $t$, choose the sign of the perturbation opposite to the endpoint selected by the rule; the same construction produces movement $n-1$.
\end{proof}

One should note that the impossibility result does not contradict Theorem~\ref{thm:main}. The construction has two
tied endpoints at the baseline, and the perturbation is chosen to favour the endpoint opposite
the reference. With $L_t=L(0)$ and reference maximizer $\{1\}$, the opposite endpoint $\{n\}$ is
also optimal and lies at matching distance $n-1$, so $\Marg(P)=0$ for every $P<n-1$ and the
hypothesis $\Marg(P)>B_t$ fails exactly when the jump occurs. A positive distance margin at
radius $P$ excludes equally optimal configurations farther than $P$ from the reference.
Theorem~\ref{thm:impossibility} therefore rules out a movement bound that is uniform over graphs
and valid for every reference and new maximizer, including under any fixed baseline
tie-breaking rule; it does not assert a discontinuity at a fixed instance whose baseline optimum
is unique and separated from distant competitors.

\section{Computing the Distance Margin and Certifying Movement}
\label{sec:margin}
We now show how to compute the distance margin $\Marg(P)$ and use it to certify a movement radius via Theorem~\ref{thm:main}.

\subsection{Formulation}
The distance margin is the old optimum minus the value of the best competitor lying farther than
$P$ from the reference; the next proposition records this, together with the monotonicity in $P$
used for certification.
\begin{proposition}[Distance margin representation]
\label{prop:far_optimum}
For every integer $P\ge0$,
\[
\Marg(P)=F_t(\Sstar)-\max_{\substack{\abs S=m\\ \dmatch(S,\Sstar)>P}}F_t(S),
\]
where the maximum is $-\infty$ if the feasible family is empty. Moreover, $\Marg(P)$ is non-decreasing in $P$.
\end{proposition}

\begin{proof}
The identity follows from
\[
\min_S\bigl(c-F_t(S)\bigr)
=
c-\max_S F_t(S)
\]
over the same finite family, with $c=F_t(\Sstar)$. As $P$ increases, the constraint $\dmatch(S,\Sstar)>P$ becomes more restrictive, so the feasible family shrinks and the maximum cannot increase. This implies that $\Marg(P)$ cannot decrease.
\end{proof}

\begin{proposition}[Perfect-matching test]
\label{prop:matching_test}
Let $\Sstar=\{q_1,\dots,q_m\}$ and $S=\{s_1,\dots,s_m\}$. Construct a bipartite graph with left vertices $q_i$, right vertices $s_j$, and an edge $(q_i,s_j)$ whenever $d(q_i,s_j)\le P$. Then
\[
\dmatch(\Sstar,S)\le P
\]
if and only if this bipartite graph has a perfect matching.
\end{proposition}

\begin{proof}If the bipartite graph has a perfect matching, then there is a bijection $\pi$ such that $d(q_i,s_{\pi(i)})\le P$ for every $i$. Thus, by Definition~\ref{def:matching_distance}, $\dmatch(\Sstar,S)\le P$.

Conversely, if $\dmatch(\Sstar,S)\le P$, then since there are finitely many bijections between $\Sstar$ and $S$, the minimum in Definition~\ref{def:matching_distance} is attained by some bijection $\pi$. Thus we have $d(q_i,s_{\pi(i)})\le P$ for every $i$, so the corresponding edges form a perfect matching in the bipartite graph.\end{proof}

This is a bipartite perfect-matching problem, which can be solved in $O(m^{2.5})$ time in the worst case \cite{lovasz2009matching}. The distance $\dmatch$ is a bottleneck assignment problem \cite{burkard2009assignment} and can be found by binary search over the $O(m^2)$ distinct pairwise distances.
\subsection{Computing the Distance Margin by Enumeration}
When all $\binom nm$ configurations can be enumerated, Proposition~\ref{prop:far_optimum} gives the following exact algorithm.
\begin{center}
\small
\begin{tabular}{@{}p{0.35cm}p{7.0cm}@{}}
\toprule
\multicolumn{2}{@{}p{7.4cm}@{}}{\textbf{Algorithm 1}: exact evaluation of $\Marg(P)$ (undersampled case)} \\
\midrule
1. & Form $\Proj_t$. \\
2. & For each $m$-subset $S$, evaluate $F_t(S)=\log\det(\Proj_t)_{SS}$, with $-\infty$ if singular. \\
3. & Record $F_t(\Sstar)=\max_SF_t(S)$ and fix a reference maximizer $\Sstar$. \\
4. & For each $S$, decide $\dmatch(S,\Sstar)>P$ by the perfect-matching test of Proposition~\ref{prop:matching_test}. \\
5. & Return $\Marg(P)=F_t(\Sstar)-\max\{F_t(S):\dmatch(S,\Sstar)>P\}$, or $+\infty$ if no such $S$ exists. \\
\bottomrule
\end{tabular}
\end{center}
If $\Proj_t$ is formed explicitly, preprocessing costs $O(n^2K)$. All pairs unweighted distances can be computed by $n$ breadth-first searches in $O\bigl(n(n+\abs{E_{\mathrm{mv}}})\bigr)$ time. Then, total enumeration will cost $O\bigl(\binom nm(m^3+m^{2.5})\bigr)$. The illustrations in Section~\ref{sec:experiments} apply Algorithm~1 by exhaustive numerical
evaluation: graph distances and matching radii are exact integers, while objective values are
computed in double precision.

\section{Losslessness of $P$-Hop and Voronoi Restrictions}
\label{sec:bridge}

Computing the distance margin by exhaustive enumeration requires $\binom{n}{m}$ objective evaluations, each costing $O(m^3)$ for the log-determinant, in addition to the matching tests above. Thus the exact far-set optimization in Proposition~\ref{prop:far_optimum} is a combinatorial subset-selection problem and exhaustive evaluation is practical only for moderate finite instances. The role of this section is therefore also computational: when the distance--margin condition holds, a $P$-hop restriction retains a global optimum and permits the subsequent optimization to be confined to a substantially smaller local feasible region.

Theorem~\ref{thm:main} places every new maximizer in the $P$-hop matching ball. We now give conditions under which these restrictions preserve the optimum for the local feasible regions used in dynamic placement.

\begin{theorem}[$P$-hop bridge]
\label{thm:bridge}
Assume the hypotheses of Theorem~\ref{thm:main} and $\Marg(P)>B_t$. Then every $S\in\argmax F_{t+1}$ satisfies
\[
\dmatch(\Sstar,S)\le P,
\]
and
\begin{equation}
  \max_{\substack{\abs S=m\\ \dmatch(S,\Sstar)\le P}}F_{t+1}(S)
  =
  \max_{\abs S=m}F_{t+1}(S).
  \label{eq:bridge_values}
\end{equation}
Moreover,
\[
\argmax_{\substack{\abs S=m\\ \dmatch(S,\Sstar)\le P}}F_{t+1}(S)
=
\argmax F_{t+1};
\]
if the global maximizer is unique, then the constrained and unconstrained maximizers agree.
\end{theorem}

\begin{proof}
Theorem~\ref{thm:main} (b) shows that every global maximizer of $F_{t+1}$ satisfies $\dmatch(S,\Sstar)\le P$. Hence the constrained family contains every global maximizer, while also being a subfamily of all $m$-sets. Its maximum therefore equals the global optimal value, which proves \eqref{eq:bridge_values}. A configuration in the constrained family attains this value if and only if it is a global maximizer, which proves the equality of the two argmax sets.
\end{proof}

\begin{proposition}[Local restrictions can lose the optimum without separation]
\label{prop:bridge_sharp}
For every integer $P\ge0$ and every $\varepsilon_0>0$, there exist connected weighted-graph Laplacians $L_t$ and $L_{t+1}$ on a common path support, with $K=2$ and $m=1$, and a reference maximizer $\Sstar$, such that
\[
\norm{L_{t+1}-L_t}_2<\varepsilon_0,
\qquad
\Marg(P)=0,
\]
while
\[
\max_{\substack{\abs S=1\\ \dmatch(S,\Sstar)\le P}}F_{t+1}(S)
<
\max_{\abs S=1}F_{t+1}(S).
\]
Moreover, the perturbation may be chosen small enough that both smallness conditions of Theorem~\ref{thm:main} hold. Thus, when the distance-margin condition fails, a $P$-hop restriction can lose the global optimum even under arbitrarily small perturbations.
\end{proposition}

\begin{proof}
Fix $P\ge0$ and choose
\[
n\ge\max\{4,P+2\}.
\]
Use the weighted-path family $L(\theta)$ constructed in Theorem~\ref{thm:impossibility}. At $\theta=0$, the one-sensor objective with $K=2$ has both endpoints $\{1\}$ and $\{n\}$ as maximizers. Set
\[
L_t:=L(0),
\qquad
\Sstar:=\{1\}.
\]
Since
\[
\dmatch(\{1\},\{n\})=n-1>P
\]
and $\{n\}$ is also a time-$t$ maximizer, it is feasible in the minimum defining $\Marg(P)$ and has zero objective loss. Hence
\[
\Marg(P)=0.
\]

For sufficiently small $\theta<0$, Theorem~\ref{thm:impossibility} shows that $\{n\}$ is the unique maximizer of the perturbed objective. Because $L(\theta)\to L(0)$ in operator norm as $\theta\to0$, we may choose such a $\theta$ with
\[
\norm{L(\theta)-L(0)}_2<\varepsilon_0.
\]
Set $L_{t+1}:=L(\theta)$. The quantities $\lstar$ and $\dstar$ at the fixed baseline are positive, whereas, by Lemma~\ref{lem:dk}, the projector drift tends to zero with $\theta$ because the
baseline eigengap $\delta_n$ is fixed. Therefore, by choosing $\abs{\theta}$ smaller if necessary, both smallness conditions of Theorem~\ref{thm:main} also hold.

Now $\{n\}$ lies outside the $P$-hop matching ball around $\Sstar$, since $n-1>P$. Because it is the unique global maximizer at time $t+1$, excluding it strictly lowers the best attainable objective:
\[
\max_{\substack{\abs S=1\\ \dmatch(S,\Sstar)\le P}}F_{t+1}(S)
<
F_{t+1}(\{n\})
=
\max_{\abs S=1}F_{t+1}(S).
\]
Thus the $P$-hop restriction genuinely loses the global optimum. \end{proof}

\subsection{Intersection with the Voronoi Region}
Nomura et al.~\cite{nomura2024dynamic} restrict sensor $i$ to the strict Voronoi cell of its current position, intersected with a $P$-hop neighborhood. We define this region and the conditions under which it retains a global optimum. Fix an ordering $\Sstar=\{q_1,\dots,q_m\}$ and define, on the fixed movement graph $G_{\mathrm{mv}}$,
\[
W_i^{<}:=
\bigl\{x\in V:\ d(x,q_i)<d(x,q_j)\ \text{for all }j\ne i\bigr\}.
\]

The strict cells are pairwise disjoint, but a vertex equidistant from two of the $q_i$ will belong to none of them. Define the corresponding local feasible family by
\begin{multline*}
V_P
:=
\bigl\{
\{x_1,\dots,x_m\}:
x_i\in W_i^{<}
\text{ and }\\ d(x_i,q_i)\le P
\text{ for every }i
\bigr\}.
\end{multline*}
Because the strict Voronoi cells are pairwise disjoint, the vertices $x_1,\dots,x_m$ in every element of $V_P$ are distinct.

As discussed in Section~\ref{sec:intro}, these results concern whether the local region contains
a global maximizer of $F_{t+1}$, not whether independent per-sensor updates find it. The regions are also centered at the reference maximizer $\Sstar$; applying them to another time-$t$ configuration requires a separate tracking-error analysis.

\begin{corollary}[Losslessness under an assignment consistent with the Voronoi cells]
\label{cor:voronoi}
Assume the hypotheses of Theorem~\ref{thm:bridge}, let $S=\{r_1,\dots,r_m\}\in\argmax F_{t+1}$, and suppose there is a permutation $\pi$ with
\begin{equation}
\begin{gathered}
  r_{\pi(i)}\in W_i^{<} \quad\text{and}\quad d(q_i,r_{\pi(i)})\le P \quad\text{for every }i.
\end{gathered}
\label{eq:compatible}
\end{equation}
Then $S\in V_P$, and consequently
\[
  \max_{S'\in V_P}F_{t+1}(S')
  =
  \max_{\abs{S'}=m}F_{t+1}(S').
\]
\end{corollary}

\begin{proof}
Condition \eqref{eq:compatible} gives $r_{\pi(i)}\in W_i^{<}$ and $d(q_i,r_{\pi(i)})\le P$ for every $i$, so $S\in V_P$. Since $S$ is a global maximizer and $V_P$ is a subfamily of all $m$-sets, maximizing over $V_P$ gives the same optimal value.
\end{proof}

For $m=1$, this condition holds automatically. For $m\ge2$, the spacing condition below is enough.

\begin{proposition}[Compatibility with strict Voronoi cells]
\label{prop:separation}
Assume $m\ge2$. If
\begin{equation}
  \min_{i\ne j}d(q_i,q_j)>2P,
  \label{eq:separation}
\end{equation}
then every vertex $x$ satisfying $d(x,q_i)\le P$ belongs to $W_i^{<}$. So, every bijection with bottleneck radius at most $P$ satisfies
\eqref{eq:compatible}.
\end{proposition}
\begin{proof}Let $x\in V$ satisfy $d(x,q_i)\le P$ and let $j\ne i$. By the triangle inequality and \eqref{eq:separation},
\begin{eqnarray*}
  d(x,q_j)&\ge& d(q_i,q_j)-d(x,q_i)\\
  &>&2P-P=P\ge d(x,q_i),
\end{eqnarray*}
so $q_i$ is strictly closer to $x$ than any other $q_j$, that is, $x\in W_i^{<}$. If $\pi$ is a bijection with $d(q_i,r_{\pi(i)})\le P$ for all $i$, then the previous argument gives $r_{\pi(i)}\in W_i^{<}$, which together with $d(q_i,r_{\pi(i)})\le P$ is \eqref{eq:compatible}.\end{proof}
\begin{corollary}[Voronoi restriction losslessness under enough spacing]
\label{cor:voronoi_sep}
Assume the hypotheses of Theorem~\ref{thm:bridge}. If $m=1$, or if $m\ge2$ and the spacing condition \eqref{eq:separation} holds, then $V_P$ contains every $S\in\argmax F_{t+1}$, and exact maximization of $F_{t+1}$ over $V_P$ is lossless in objective value.
\end{corollary}
\begin{proof}For $m=1$, $W_1^{<}=V$. For $m\ge2$, Theorem~\ref{thm:bridge} supplies, for each $S\in\argmax F_{t+1}$, a bijection from $\Sstar$ to $S$ with bottleneck radius at most $P$. Proposition~\ref{prop:separation} makes it compatible in the sense of \eqref{eq:compatible}, and Corollary~\ref{cor:voronoi} applies.\end{proof}

A fixed assignment of Voronoi ties does not affect Corollary~\ref{cor:voronoi_sep}, since \eqref{eq:separation} guarantees uniqueness within distance $P$;  condition \eqref{eq:separation} is sufficient but not necessary for Corollary~\ref{cor:voronoi}.
\section{Extensions}
\label{sec:extensions}
We extend the main result to oversampled and regularized D-optimality and then note how learned dictionaries fit directly into the general projector framework without further proofs.
\subsection{Oversampled D-Optimality, $m\ge K$}
\label{subsec:oversampled}
In this subsection $\Sstar$ denotes a maximizer of the oversampled objective $F_t^{\mathrm{ov}}$;
it need not coincide with the reference used for $F_t$ in the undersampled setting.
For every $m$-set $S$ one has $0\preceq H_t(S)\preceq I_K$. Consequently, if $\det H_t(S)>0$ then $\lmin(H_t(S))\ge\det H_t(S)=e^{\Fov(S)}$. Since $0\preceq D_S\preceq I_n$ and $U_t^\top U_t=I_K$, congruence gives $0\preceq U_t^\top D_SU_t\preceq U_t^\top U_t=I_K$, so all eigenvalues lie in $[0,1]$ and the determinant argument of Proposition~\ref{prop:auto} applies.
\begin{lemma}[Oversampled log-determinant stability]
\label{lem:over_logdet}

Let $c_S:=\lmin(H_t(S))>0$ and let $r\ge\rho_t$ satisfy $r\le c_S/4$. Then $H_{t+1}(S)$ is positive definite and
\begin{equation}
  \abs{\Fov(S)-F_{t+1}^{\mathrm{ov}}(S)}\le\frac{4Kr}{c_S}.
  \label{eq:over_logdet}
\end{equation}
\end{lemma}
\begin{proof}Set $A_t:=\Proj_tD_S\Proj_t=U_tH_t(S)U_t^\top$, so the nonzero eigenvalues of $A_t$ are the same as those of $H_t(S)$, and likewise at time $t+1$. From
\[
  A_t-A_{t+1}=(\Proj_t-\Proj_{t+1})D_S\Proj_t+\Proj_{t+1}D_S(\Proj_t-\Proj_{t+1})
\]
and $\norm{D_S}_2\le1$, $\norm{\Proj_t}_2=\norm{\Proj_{t+1}}_2=1$ we get $\norm{A_t-A_{t+1}}_2\le2\rho_t\le2r$. Order eigenvalues nonincreasingly. By Weyl's inequality, $\abs{\lambda_j(A_{t+1})-\lambda_j(A_t)}\le2r$ for all $j$. For $1\le j\le K$ we have $\lambda_j(A_t)\ge c_S$ and $r\le c_S/4$, so $2r\le c_S/2$, so $\lambda_j(A_{t+1})\ge c_S/2>0.$ Since $\operatorname{rank}A_{t+1}\le K$, these are exactly the nonzero eigenvalues of $A_{t+1}$, hence the eigenvalues of $H_{t+1}(S)$, which is thus positive definite. Setting
\[  x_j:=\frac{\lambda_j(A_{t+1})-\lambda_j(A_t)}{\lambda_j(A_t)}, \text{ and }\abs{x_j}\le\frac{2r}{c_S}\le\frac12,\]
and using $\det H_t(S)=\prod_{j\le K}\lambda_j(A_t)$, $\det H_{t+1}(S)=\prod_{j\le K}\lambda_j(A_{t+1})$ together with \eqref{eq:scalar_log},
\begin{eqnarray*}
  \abs{F_{t+1}^{\mathrm{ov}}(S)-\Fov(S)}&=&\Bigl|\sum_{j=1}^K\log(1+x_j)\Bigr| \\
  &\le&2\sum_{j=1}^K\abs{x_j}\le\frac{4Kr}{c_S}.
\end{eqnarray*} This proves \eqref{eq:over_logdet}.\end{proof}

Since only $\norm{\Proj_t-\Proj_{t+1}}_2$ enters, Lemma~\ref{lem:over_logdet} also holds with
the times $t$ and $t+1$ exchanged, that is, with $c_S:=\lmin(H_{t+1}(S))$.
 Let $\Sstar\in\argmax\Fov$, set $\dstarov:=\det H_t(\Sstar)$, and define the oversampled distance margin by
\[
  \MargOver(P):=\min_{\substack{\abs S=m\\ \dmatch(S,\Sstar)>P}}\bigl[\Fov(\Sstar)-\Fov(S)\bigr].
\]

Since $U_t$ has rank $K$, it has $K$ linearly independent rows. For $m\ge K$, include those rows in an $m$-set $S$; then $H_t(S)\succ0$. Therefore $F_t^{\mathrm{ov}}$ has a finite maximum, and every reference maximizer has positive determinant.

\begin{theorem}[Movement bound in the oversampled case]
\label{thm:over}

Assume $m\ge K$ and Assumption~\ref{ass:movement}. Let $r\ge\rho_t$ satisfy
\begin{equation}
\begin{gathered}
  r\le\frac{\dstarov}{8},~ {\rm and}~
  \etaov:=\frac{4Kr}{\dstarov}\le\log2 .
\end{gathered}
\label{eq:over_smallness}
\end{equation}
Then every $S\in\argmax F_{t+1}^{\mathrm{ov}}$ satisfies
\[
\Fov(\Sstar)-\Fov(S)\le3\etaov.
\]
Consequently, for every integer $P\ge0$, one has that
\begin{equation}
\begin{gathered}
  \MargOver(P)>3\etaov=\frac{12Kr}{\dstarov}\\
  \Longrightarrow
  \dmatch(\Sstar,S)\le P
  \quad\\
  \text{for every }S\in\argmax F_{t+1}^{\mathrm{ov}}.
\end{gathered}
\label{eq:over_implication}
\end{equation}
Under Assumption~\ref{ass:spectral} one may take $r=4\varepsilon/\delta$, in which case the hypotheses read $4\varepsilon/\delta\le\dstarov/8$ and $16K\varepsilon/(\dstarov\delta)\le\log2$, and the threshold becomes $48K\varepsilon/(\dstarov\delta)$.
\end{theorem}

 \begin{proof}
Since $0\preceq H_t(\Sstar)\preceq I_K$, we have $\lmin(H_t(\Sstar))\ge\det H_t(\Sstar)=\dstarov$,
so $r\le\dstarov/8\le\lmin(H_t(\Sstar))/4$. Lemma~\ref{lem:over_logdet} at $S=\Sstar$ gives
\[
\abs{F_{t+1}^{\mathrm{ov}}(\Sstar)-\Fov(\Sstar)}\le\frac{4Kr}{\dstarov}=\etaov .
\]
Let $S\in\argmax F_{t+1}^{\mathrm{ov}}$. By optimality and this bound,
$F_{t+1}^{\mathrm{ov}}(S)\ge F_{t+1}^{\mathrm{ov}}(\Sstar)\ge\log\dstarov-\etaov$, so
$\det H_{t+1}(S)\ge\dstarov e^{-\etaov}\ge\dstarov/2$ because $\etaov\le\log2$. Since
$0\preceq H_{t+1}(S)\preceq I_K$, one has that
\[
\lmin(H_{t+1}(S))\ge\frac{\dstarov}{2},
\] and thus
\[r\le\frac{\dstarov}{8}\le\frac{\lmin(H_{t+1}(S))}{4}.
\]
Lemma~\ref{lem:over_logdet} with the two times exchanged therefore gives
\begin{eqnarray}
\abs{\Fov(S)-F_{t+1}^{\mathrm{ov}}(S)}&\le&\frac{4Kr}{\lmin(H_{t+1}(S))}\nonumber\\&\le&\frac{8Kr}{\dstarov}=2\etaov .\nonumber
\end{eqnarray}
Combining the two bounds with optimality of $S$ at time $t+1$,
\begin{align*}
\Fov(\Sstar)-\Fov(S)
&=\bigl[\Fov(\Sstar)-F_{t+1}^{\mathrm{ov}}(\Sstar)\bigr]\\
&\quad+\bigl[F_{t+1}^{\mathrm{ov}}(\Sstar)-F_{t+1}^{\mathrm{ov}}(S)\bigr]\\
&\quad+\bigl[F_{t+1}^{\mathrm{ov}}(S)-\Fov(S)\bigr]\\
&\le\etaov+0+2\etaov=3\etaov .
\end{align*}
If $\dmatch(\Sstar,S)>P$, then $S$ is feasible in the minimum defining $\MargOver(P)$, so
$\Fov(\Sstar)-\Fov(S)\ge\MargOver(P)>3\etaov$, a contradiction; this proves
\eqref{eq:over_implication}. Under Assumption~\ref{ass:spectral}, Lemma~\ref{lem:dk} shows that
$r=4\varepsilon/\delta$ satisfies $r\ge\rho_t$, which gives the Laplacian specialization.
\end{proof}\smallskip

\subsection{Regularized D-Optimality}
\label{subsec:regularized}

For this subsection, define
\[
C_s(S):=R_S\Proj_sR_S^\top,
\qquad s\in\{t,t+1\},
\]
for arbitrary $m$. Lemma~\ref{lem:compression} still applies because its proof does not use $m\le K$. For $\tau>0$ set
\[
F_s^\tau(S)
:=
\log\det\bigl(C_s(S)+\tau I_m\bigr).
\]
With $U_{s,S}=R_SU_s$ as in Section~\ref{subsec:setup}, Sylvester's determinant identity gives
\[\det(\tau I_m + U_{s,S}U^\top_{s, S}) = \tau^m \det(I_K + \tau^{-1}U^\top_{s,S}U_{s, S})\]
and therefore
\[
  \det\bigl(C_s(S)+\tau I_m\bigr)
  =
  \tau^{m-K}
  \det\bigl(U_{s,S}^\top U_{s,S}+\tau I_K\bigr).
\]
Thus $F_s^\tau(S)$ and the regularized $K\times K$ information log-determinant differ only by constant $(m-K)\log\tau$, so they have the same maximizers. Since $C_s(S)\succeq0$ and $\tau>0$, $C_s(S)+\tau I_m$ is positive definite for every $S$. Let $S_t^{\star,\tau}\in\argmax F_t^\tau$ be a reference maximizer and define
\[
  \MargTau(P):=\min_{\substack{\abs S=m\\ \dmatch(S,S_t^{\star,\tau})>P}}\bigl[F_t^\tau(S_t^{\star,\tau})-F_t^\tau(S)\bigr].
\]
\begin{lemma}[Uniform regularized stability]
\label{lem:reg}
Let $r\ge\rho_t$ with $r/\tau\le1/2$. Then for every $m$-set $S$,
\begin{equation}
  \abs{F_t^\tau(S)-F_{t+1}^\tau(S)}\le\frac{2mr}{\tau}=:\eta_t^\tau .
  \label{eq:reg_bound}
\end{equation}
\end{lemma}
\begin{proof}Apply Lemma~\ref{lem:logdet} with $X=C_t(S)+\tau I_m$ and $Y=C_{t+1}(S)+\tau I_m$. Then $\lmin(X)\ge\tau$ because $C_t(S)\succeq0$, and $\norm{Y-X}_2=\norm{C_t(S)-C_{t+1}(S)}_2\le\rho_t\le r\le\tau/2\le\lmin(X)/2$ by Lemma~\ref{lem:compression}. The bound $2mr/\tau$ follows.\end{proof}
\begin{theorem}[Movement bound for the regularized objective]
\label{thm:reg}
Assume Assumption~\ref{ass:movement}.
Let $r\ge\rho_t$ with $r/\tau\le1/2$. Then for every integer $P\ge0$, and for every $S\in\argmax F_{t+1}^\tau$, we have
\begin{equation}
\begin{gathered}
  \MargTau(P)>\frac{4mr}{\tau}
  \Longrightarrow\quad \dmatch\bigl(S_t^{\star,\tau},S\bigr)\le P.
\end{gathered}
\label{eq:reg_implication}
\end{equation}
Under Assumption~\ref{ass:spectral} it suffices that $4\varepsilon/(\tau\delta)\le1/2$ and $\MargTau(P)>16m\varepsilon/(\tau\delta)$.
\end{theorem}
\begin{proof}Let $S\in\argmax F_{t+1}^\tau$. By Lemma~\ref{lem:reg} applied at $S_t^{\star,\tau}$ and at $S$, and by optimality of $S$ at time $t+1$,
\begin{align*}
  F_t^\tau(S_t^{\star,\tau})-F_t^\tau(S)
  &=\bigl[F_t^\tau(S_t^{\star,\tau})-F_{t+1}^\tau(S_t^{\star,\tau})\bigr]\\
  &\quad+\bigl[F_{t+1}^\tau(S_t^{\star,\tau})-F_{t+1}^\tau(S)\bigr]\\
  &\quad+\bigl[F_{t+1}^\tau(S)-F_t^\tau(S)\bigr]\\
  &\le\eta_t^\tau+0+\eta_t^\tau=2\eta_t^\tau=\frac{4mr}{\tau}.
\end{align*}
If $\dmatch(S_t^{\star,\tau},S)>P$ the margin definition forces the opposite strict inequality. The Laplacian statement follows from Lemma~\ref{lem:dk} with $r=4\varepsilon/\delta$.\end{proof}

Note that the threshold here is $2\eta_t^\tau$: Lemma~\ref{lem:reg} is uniform over configurations, so no separate conditioning argument at the new maximizer is needed.

\subsection{Dictionaries}
\label{subsec:dictionary}

The projector framework also covers signal models given by a learned dictionary. Let
$\Psi_t,\Psi_{t+1}\in\R^{n\times K}$ have full column rank; in the setting of
\cite{nomura2024dynamic}, $\Psi_s$ is the learned dictionary after application of a graph filter
(see below). For $s\in\{t,t+1\}$ set
\begin{eqnarray}
Q_s&:=&\Psi_s(\Psi_s^\top\Psi_s)^{-1}\Psi_s^\top=U_sU_s^\top,
\nonumber\\
U_s&:=&\Psi_s(\Psi_s^\top\Psi_s)^{-1/2}.\nonumber
\end{eqnarray}
Thus $Q_s$ is the orthogonal projector onto the column span of $\Psi_s$, and $U_s$ is the
orthonormalized dictionary. The results of Section~\ref{sec:main} apply verbatim with
$\Proj_s=Q_s$ and $\rho_t=\norm{Q_t-Q_{t+1}}_2$, which can be computed directly from the two
dictionaries; no eigengap or Davis--Kahan estimate is needed.

We shall show here precisely which criterion this covers: for $m\le K$ the objective is
$F_s(S)=\log\det(Q_s)_{SS}$, the row-volume criterion for the orthonormalized dictionary $U_s$,
and in general this differs from row-volume selection on $\Psi_s$ itself. For example, with
$n=3$, $K=2$ and dictionary rows $(10,0)$, $(0,1)$, $(0,0)$, raw row volumes for $m=1$ single
out the first row, whereas the diagonal of $Q_s$ is $(1,1,0)$ and the first two rows tie. The
two criteria agree up to a constant when $m=K$, since then
$\det(Q_s)_{SS}=\det(\Psi_{s,S})^2/\det(\Psi_s^\top\Psi_s)$ and the denominator does not depend
on $S$; for $m>K$ the information-matrix objective has the same invariance, since
$\det(U_s^\top D_SU_s)=\det(\Psi_s^\top D_S\Psi_s)/\det(\Psi_s^\top\Psi_s)$.

Recall that in \cite{nomura2024dynamic}   a dictionary $A_t$ is
learned online, a graph low-pass filter $G$ models sensor coverage, and sensors are chosen by
maximizing $\det\bigl(R_SGA_tA_t^\top GR_S^\top\bigr)$, that is, the raw row volume of
$\Psi_t:=GA_t$ \cite[eqs.~(10), (24)--(25)]{nomura2024dynamic}. When $GA_t$ has full column
rank, our results apply with $Q_t$ the projector onto its column span. The two criteria coincide
up to an $S$-independent factor when the number of sensors equals the dictionary rank, but they
can select different sets when there are fewer sensors than atoms, as in the synthetic
experiments of \cite{nomura2024dynamic}, by the example above. Moreover, the maximization there
is approximate: a greedy Schur-complement step is solved independently for each sensor over its
strict Voronoi cell intersected with a $P$-hop neighbourhood
\cite[eq.~(29) and Algorithm~2]{nomura2024dynamic}. Our results therefore certify the movement of
exact maximizers of the projector criterion associated with the same learned subspace; they do
not certify the placements produced by that algorithm.

\section{Experiments on Synthetic Graphs}
\label{sec:experiments}
The experiments illustrate the certificate rather than benchmark it. They address two
questions: how conservative the certified radius is compared with the movement that actually
occurs (Section~\ref{subsec:exp_certificates}), and how the distance margin varies with the
radius in a representative instance (Section~\ref{subsec:exp_landscape}).
\subsection{Protocol and Setup}
\label{subsec:protocol}

The following synthetic experiments use connected weighted graphs with a fixed edge set. The corresponding unweighted graph is also the movement graph, so perturbations change only the edge weights, not the distances used by $\dmatch$. Together, three graph families cover different geometrical settings (regular lattice, irregular spatial graph, and one with community structure). We use $K=4$ and $m=2$ to stay in the undersampled setting of Theorem~\ref{thm:main} while keeping the number of configurations small enough to compute all directly.

We use two baseline graphs from each of three families (with random but fixed seeds): $4\times4$ grids ($n=16$; seeds $500$ and $537$), random geometric graphs ($n=12$, connection radius $0.46$; seeds $509$ and $546$), and two-block stochastic block graphs ($n=12$, block sizes $6$ and $6$, intra-block connection probability $0.50$, inter-block connection probability $0.08$; seeds $503$ and $540$). If a random geometric graph is disconnected, we connect its components using the shortest Euclidean edges between components. If a stochastic block graph is disconnected, we order its connected components by their smallest vertex label, connect the smallest-labeled vertices of the first two components, and repeat until the graph is connected. Baseline edge weights are drawn independently from $\mathrm{Unif}[0.7,1.3]$. For each baseline, the listed seed is used both by the random graph generator, when the topology is random, and to initialize the NumPy generator used for the baseline edge weights; for the deterministic grid topology it therefore controls only the edge weights.

For each edge with weight $w$, we replace $w$ by $(1+\alpha)w$ for
\[
\alpha\in\{\pm0.02,\pm0.05,\pm0.10,\pm0.20\}.
\]
while keeping all other edge weights fixed. These values test small and moderate single edge changes while keeping all weights positive and the edge set fixed.

Each perturbation case modifies exactly one edge of a baseline graph; changes are not
accumulated, so every case compares a baseline Laplacian $L_t$ with a single perturbed Laplacian
$L_{t+1}$. For each Laplacian we form the rank-$K$ projector onto its first $K$ eigenvectors and
evaluate the objective on all $\binom nm$ configurations. Graph distances and bottleneck matching
radii are exact integers; eigenvectors, log-determinants, and derived spectral quantities are
computed in double precision. A compressed matrix is treated as positive definite when its
computed minimum eigenvalue exceeds $10^{-11}$.

Since exact ties cannot be detected in floating point, we work with numerical maximizer
families. For $s\in\{t,t+1\}$ let $F_s^{\max}:=\max_{\abs S=m}F_s(S)$ over all $m$-subsets, and
\begin{eqnarray}
\widehat{\mathcal S}_s&:=&\bigl\{S:\abs S=m,\ F_s(S)\ge F_s^{\max}-\tau_s\bigr\},
\nonumber\\
\tau_s&:=&10^{-9}+10^{-9}\max\{1,\abs{F_s^{\max}}\}.\nonumber
\end{eqnarray}
The reference $\Sstar$ is the lexicographically smallest element of $\widehat{\mathcal S}_t$.
% CHECK: confirm the reference rule used in the code
Objective gaps of magnitude below $10^{-10}$ are set to zero when computing $\Marg(P)$.

We take $r=\rho_t$, computed from the two projectors, and call a case eligible when both smallness conditions of Theorem~\ref{thm:main} hold up to an absolute tolerance of $10^{-12}$. For eligible cases define
\begin{eqnarray}
P_{\mathrm{cert}}&:=&\min\{P:\Marg(P)>B_t\},
\nonumber\\
P_{\mathrm{act}}&:=&\max_{S\in\widehat{\mathcal S}_{t+1}}\dmatch(\Sstar,S),\nonumber\\
P_{\max}&:=&\max_{\abs S=m}\dmatch(\Sstar,S).
\end{eqnarray}
The numerical margins are comfortably separated from the stated tolerances. Across the eligible cases, the minimum computed values of $\lambda_{\min}(C_t(S_t^\star))$ and of $\lambda_{\min}(C_{t+1}(S))$ over the retained new maximizers are $3.555\times10^{-1}$ and $3.499\times10^{-1}$, respectively. The minimum slacks in the two smallness conditions are $1.178\times10^{-1}$ and $1.316\times10^{-3}$; the minimum finite crossing and preceding slacks are $2.109\times10^{-3}$ and $3.248\times10^{-3}$. The minimum slack between $B_t$ and the largest
old-objective loss among retained new maximizers
is $1.826\times10^{-3}$. In particular, no case is classified as eligible solely because of the $10^{-12}$ comparison tolerance. 

For exact arithmetic and exact maximizers,
Theorem~\ref{thm:main} gives $P_{\mathrm{act}}\le P_{\mathrm{cert}}$. Members of
$\widehat{\mathcal S}_{t+1}$ are maximizers only up to $\tau_{t+1}$, so the comparisons reported
below are numerical observations rather than rigorous floating-point certificates. The radius
$P_{\max}$ contains every configuration geometrically; when the smallness conditions hold it
gives a vacuous certificate, and we call a certificate nontrivial when
$P_{\mathrm{cert}}<P_{\max}$.

All random seeds are fixed. Calculations use Python, NumPy, SciPy, pandas, and NetworkX, and figures are generated with Matplotlib. Reproducibility code and the CSV files used for the figures are available in the GitHub repository \url{https://github.com/sdg-bella/movement-stability-of-sensors-under-perturbations}. The scripts, CSV outputs, and figures used for the reported
results are archived in release
\href{https://github.com/sdg-bella/movement-stability-of-sensors-under-perturbations/releases/tag/arxiv-v1}
{\texttt{arxiv-v1}},
corresponding to commit
\href{https://github.com/sdg-bella/movement-stability-of-sensors-under-perturbations/tree/b59f808048a7c645e6ea7c989d26f02b906f3940}
{\texttt{b59f808}}.

\begin{figure}[h!]
\centering
\includegraphics[width=\linewidth]{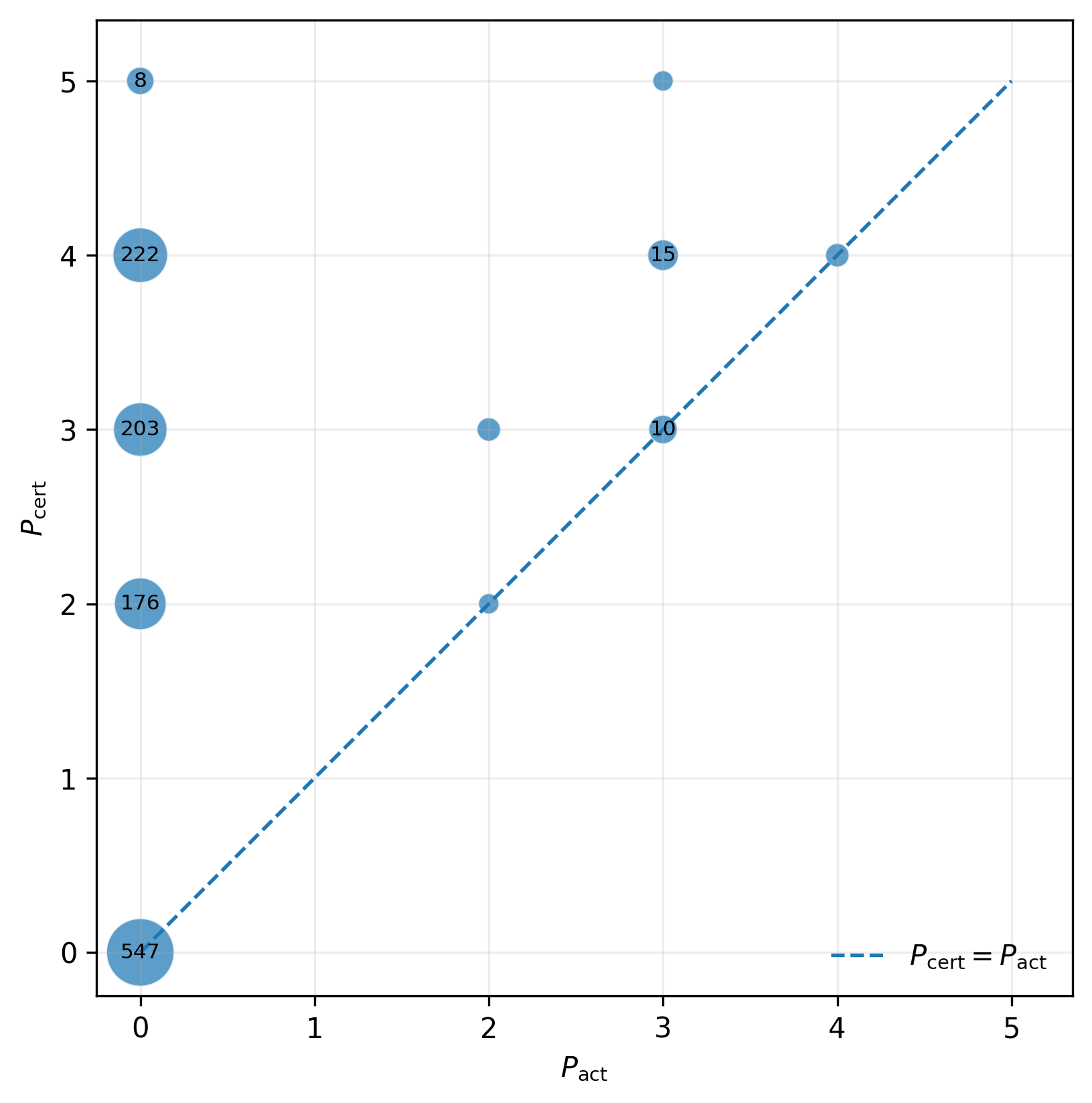}
\caption{Actual matching movement $P_{\mathrm{act}}$ versus certified radius
$P_{\mathrm{cert}}$ over the 1,189 eligible perturbation cases. Each bubble marks an integer pair
$(P_{\mathrm{act}},P_{\mathrm{cert}})$; its area is proportional to the number of cases with that
pair, and the printed number is that count. The dashed diagonal is
$P_{\mathrm{cert}}=P_{\mathrm{act}}$: points on it are sharp certificates, and no point lies
below it.}
\label{fig:pact_pcert}
\end{figure}

\subsection{Certifying Optimizer Movement on Synthetic Graphs}
\label{subsec:exp_certificates}

\begin{figure*}[t!]
\centering
\includegraphics[width=0.95\linewidth]{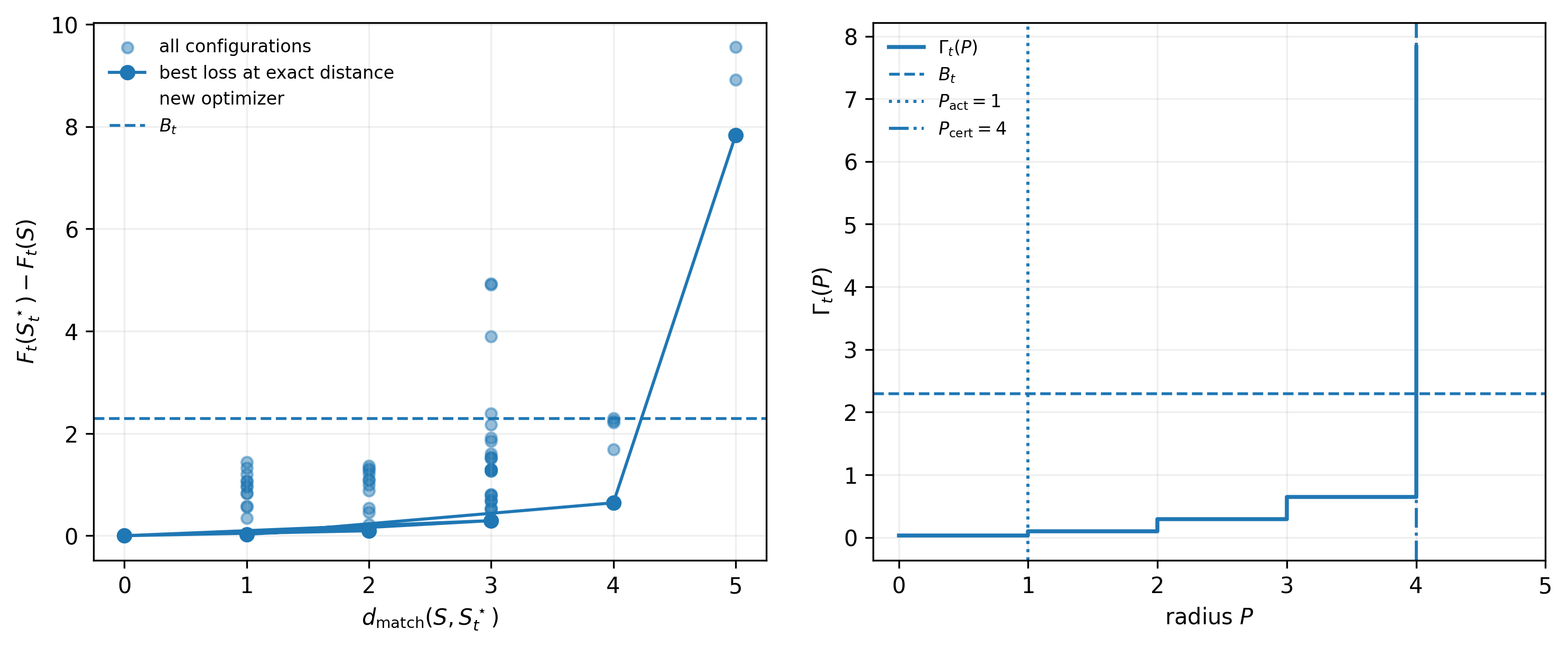}
\caption{Old-objective landscape for the illustrative random geometric instance (left) and the
induced distance-margin curve (right). The horizontal line is $B_t$; vertical lines indicate
$P_{\mathrm{act}}$ and $P_{\mathrm{cert}}$. This is separate from the sweep summarized in Figure~\ref{fig:pact_pcert}.
}
\label{fig:landscape_margin}
\end{figure*}
The methodology produces 1,224 one-edge perturbation cases. Of these, 1,189 (97.1\%) satisfy both projector perturbation smallness conditions with $r=\rho_t$ and are eligible for certification. In every eligible case, the computation gives $P_{\mathrm{act}}\le P_{\mathrm{cert}}$, consistent with the exact-arithmetic conclusion of Theorem~\ref{thm:main}.
A comparison between  the observed movement
$P_{\mathrm{act}}$ and the certified radius
$P_{\mathrm{cert}}$ across all eligible perturbation cases is given in Figure~\ref{fig:pact_pcert}.
Points on the diagonal correspond to equality of the two
radii, while points above it show how conservative the
certificate is; bubble sizes indicate the number of cases
with each pair of radii.

Among the eligible cases, 1,035 (87.0\%) have a nontrivial certificate $P_{\mathrm{cert}}<P_{\max}$. In 561 cases (47.2\%), $P_{\mathrm{cert}}=P_{\mathrm{act}}$ with integer radius, with 547 of these being $P_{\mathrm{cert}}=P_{\mathrm{act}}=0$. The optimizer moves in 33 eligible cases, and 14 of these are integer-radius sharp. The median difference $P_{\mathrm{cert}}-P_{\mathrm{act}}$
is two hops over all eligible cases and one hop among
the $33$ cases with positive movement.
The two-hop median gap also shows that it can still be conservative. It often rules out large movements without giving the smallest possible radius. The six baselines are not matched across graph families by projector drift or distance margin, so these frequencies should not be used to rank the graph families.

\subsection{Distance Margin Curve}
\label{subsec:exp_landscape}

We show the distance margin using a separate illustrative random geometric graph with $n=12$, $K=4$, and $m=2$ (seed $911$), selected by the fixed rule in the reproducibility code; this instance is not part of the six-baseline sweep in Section~\ref{subsec:exp_certificates}. The selected perturbation increases the weight of edge $(3,5)$ by $10\%$, using the zero-based vertex labels of the code. For this perturbation,
\[
P_{\mathrm{act}}=1,\qquad
P_{\mathrm{cert}}=4,\qquad
P_{\max}=5,
\]
with $\rho_t=0.05193$ and $B_t=2.30128$; this can be seen in Figure~\ref{fig:landscape_margin}, which plots $\dmatch(S,\Sstar)$ versus $F_t(\Sstar)-F_t(S)$ for every configuration $S$. The left plot therefore shows the objective loss of every configuration. For each integer radius $P$, $\Marg(P)$ is the smallest objective loss among configurations with $\dmatch(S,\Sstar)>P$.

The right side shows how $\Marg(P)$ changes with $P$, and the first integer radius satisfying $\Marg(P)>B_t$ is $P_{\mathrm{cert}}$. It is important to highlight that the left panel shows what a single second best gap misses. Configurations with nearly the same objective loss can lie at very different matching distances. The distance margin keeps this distance information, which is what makes it useful for certifying optimizer movement.

\section{Conclusions}
\label{sec:conclusion}
We developed a geometric stability framework for dynamic
determinant-based subset selection, measuring sensor movement
by bottleneck matching on a fixed mobility graph. The main
findings are:
\begin{itemize}
\item Small projector perturbations bound the movement of
every new maximizer when the distance margin exceeds the
objective perturbation budget.

\item Perturbation size and eigengap alone cannot guarantee
small movement, even on a positively weighted path.

\item Distance margins certify when local restrictions
retain the global optimum, with an additional compatibility
condition for Voronoi restrictions.

\item The framework extends to oversampled and regularized
objectives and to learned-subspace projectors; under suitable
ensemble assumptions, the second-best gap controls the
probability of optimizer change.
\end{itemize}

The numerical illustrations show that the certificate,
evaluated exhaustively in double precision, is nontrivial
in most of the tested cases, although its radius can be
conservative. Positive movement occurs in $33$ of the
$1{,}189$ numerically eligible cases. These observations
motivate studying the lower tail of the second-best
objective gap in concrete dynamic graph models, which
Corollary~\ref{cor:average_case} relates to the probability
of optimizer change under its stated assumptions.

The framework may be useful for mobile sensing on networks,
where an updated signal model can require sensors to move
between candidate locations. Examples include environmental
monitoring and robotic sensing, provided that movement can
be represented by a fixed mobility graph and placement is
governed by one of the determinant criteria studied here.
The certificate could help determine when a local
reconfiguration retains a global optimum. For larger
networks, an important next step is to obtain conservative
lower bounds on the distance margin without enumerating
all configurations, making the certificate useful beyond
the small instances considered here.

Further extensions concern the sensors and their update
rules. The present matching distance treats sensors as
interchangeable and does not account for interactions
between their trajectories. Incorporating heterogeneous
capabilities, travel costs, and collision constraints would
require adapting the configuration space and its movement
geometry. Another direction is to extend the stability
results to approximate maximizers, connecting the present
guarantees for exact optima to configurations produced by
practical placement algorithms.

\section*{Declarations}
\noindent\textbf{Acknowledgements.} 
 LPS is grateful to Magdalen College and the Mathematical Institute at the University of
Oxford for their hospitality and support during the Visiting Fellowship in the course of
which this work was completed. 
 
\smallbreak
\noindent\textbf{Competing interests.} The authors have no relevant financial or non-financial interests to disclose.
\smallbreak
\noindent\textbf{Data and Code availability.} All data reported in this paper are generated synthetically by the procedures described in Section~\ref{sec:experiments}.   Our code and CSV files  are available at \url{https://github.com/sdg-bella/movement-stability-of-sensors-under-perturbations}.
\smallbreak
\noindent\textbf{Author contributions.}
Conceptualization: Isabella Yin, Laura P. Schaposnik; Methodology and formal analysis: Isabella Yin; Software and numerical investigation: Isabella Yin; Writing: Isabella Yin,  Laura P. Schaposnik; Supervision: Laura P. Schaposnik. Both authors read and approved the final manuscript.
\smallbreak
\noindent\textbf{AI-assisted technologies.} ChatGPT and Claude   were used during manuscript preparation to assist with language refinement. All AI-assisted output was reviewed and edited by the authors, taking responsibility for the final manuscript.

\smallbreak
\noindent\textbf{Funding.} The research of LPS is partially supported by NSF FRG Award DMS2152107 and NSF CAREER
Award DMS 1749013, and also supported in part by grants from the NSF (DMS2235451) and
Simons Foundation (MP-TMPS-00005320) to the NSF-Simons National Institute for Theory and
Mathematics in Biology (NITMB), and by a Simons Travel Award.

\pagebreak
\appendix \section{Notation}
\label{app:notation}
We collect the main symbols used in the paper.
\begin{table}[H]
\caption{Main notation (I): setting and objectives.}
\smallbreak
\label{tab:notation}
\centering
\small
\begin{tabular}{@{}p{1.8cm}p{5.8cm}@{}}
\toprule
Symbol & Meaning \\
\midrule
$V$, $n$ & vertex set, number of vertices $n=\abs V$ \\

$K$, $m$ & projector rank and number of selected coordinates \\

$S$ & configuration, i.e.\ an $m$-subset of $V$ \\

$G_{\mathrm{mv}}$, $d$ & fixed mobility graph and its shortest-path distance \\

$\Proj_t$, $U_t$ & rank-$K$ orthogonal projector at time $t$ and an orthonormal basis of its range \\

$\rho_t$ & projector drift $\norm{\Proj_t-\Proj_{t+1}}_2$ \\

$R_S$, $D_S$ & row-selector matrix and associated diagonal selector $D_S=R_S^\top R_S$ \\

$C_t(S)$ & compressed projector $(\Proj_t)_{SS}$, used in the undersampled case \\

$H_t(S)$ & information matrix $U_t^\top D_SU_t$, used in the oversampled case \\

$F_t$, $F_t^{\mathrm{ov}}$ & undersampled and oversampled log-determinant objectives \\
\bottomrule
\end{tabular}
\end{table}
\begin{table}[H]
\caption{Main notation (II): movement, margins, and certification.}
\smallbreak
\label{tab:notation2}
\centering
\small
\begin{tabular}{@{}p{1.4cm}p{6cm}@{}}
\toprule
Symbol & Meaning \\
\midrule
$\Sstar$ & fixed reference maximizer in $\argmax F_t$ \\

$\dstar$, $\lstar$ & determinant and minimum eigenvalue of $C_t(\Sstar)$ \\

$P$ & nonnegative integer movement radius, measured in graph hops \\

$\dmatch$ & bottleneck matching distance between configurations \\

$\Marg(P)$ & distance margin from configurations more than $P$ hops from $\Sstar$ \\

$r$ & certified upper bound on projector drift, $r\ge\rho_t$ \\

$\eta_{t,1}$, $\eta_{t,2}$, $B_t$ & perturbation terms appearing in Theorem~\ref{thm:main}, with $B_t=\eta_{t,1}+\eta_{t,2}$ \\

$\varepsilon$, $\delta$ &
Laplacian perturbation size $\norm{L_t-L_{t+1}}_2$ and baseline eigengap
$\lambda_{K+1}(L_t)-\lambda_K(L_t)$ \\

$P_{\rm act}$ & largest matching distance from $S_t^\star$ to a
retained numerical maximizer in $\widehat{\mathcal S}_{t+1}$\\

$P_{\rm cert}$ & smallest movement radius certified by the theorem \\

$P_{\max}$ & largest matching distance from $\Sstar$ to any $m$-subset \\
\bottomrule
\end{tabular}
\end{table}

\begin{table}[H]
\caption{Main notation (III): extensions and local restrictions.}
\smallbreak
\label{tab:notation3}
\centering
\small
\begin{tabular}{@{}p{0.24\columnwidth}p{0.72\columnwidth}@{}}
\toprule
Symbol & Meaning \\
\midrule
$\dstarov$, $\etaov$ & determinant of $H_t(\Sstar)$ for a maximizer $\Sstar$ of $F_t^{\mathrm{ov}}$, and the oversampled perturbation term \\
$\MargOver(P)$ & oversampled distance margin \\
$\tau$ & regularization parameter \\
$S_t^{\star,\tau}$, $\MargTau(P)$, $\eta_t^\tau$ & regularized reference maximizer, distance margin, and perturbation term \\
$\Psi_s$, $Q_s$ & (filtered) learned dictionary and the projector onto its span \\
$W_i^{<}$ & strict Voronoi cell of $q_i$ in $G_{\mathrm{mv}}$ \\
$V_P$ & Voronoi-restricted local feasible family \\
\bottomrule
\end{tabular}
\end{table}

\newpage
\small
\bibliographystyle{plain}
\bibliography{references}

\end{document}